\documentclass[12pt, letterpaper, oneside, hidelinks]{amsart}
\usepackage{fullpage} 
\usepackage{indentfirst}
\usepackage[margin=1in]{geometry}
\usepackage{amssymb}
\usepackage{amsmath}
\usepackage{amsthm}
\usepackage{mathtools}
\usepackage{stmaryrd}
\usepackage{tikz-cd}
\usepackage{enumitem}
\usepackage{float}
\numberwithin{equation}{section}
\usepackage{etoolbox}
\patchcmd{\thebibliography}{\chapter*}{\section*}{}{}
\usepackage[utf8]{inputenc}
\newtheorem{lemma}{Lemma}
\numberwithin{lemma}{section}
\newtheorem{proposition}[lemma]{Proposition}
\newtheorem{theorem}[lemma]{Theorem}
\newtheorem{corollary}[lemma]{Corollary}
\theoremstyle{definition}
\newtheorem{definition}[lemma]{Definition}
\newtheorem*{remark}{Remark}

\let\mod=\undefined
\DeclareMathOperator{\cone}{cone}
\DeclareMathOperator{\pd}{pd}

\DeclareMathOperator{\per}{per}
\DeclareMathOperator{\proj}{proj}
\DeclareMathOperator{\id}{id}

\DeclareMathOperator{\Hom}{Hom}

\DeclareMathOperator{\mod}{mod}
\DeclareMathOperator{\gldim}{gl.dim}
\DeclareMathOperator{\daut}{DAut}
\DeclareMathOperator{\im}{Im}
\DeclareMathOperator{\Ext}{Ext}
\DeclarePairedDelimiter\floor{\lfloor}{\rfloor}
\DeclarePairedDelimiter\ceil{\lceil}{\rceil}
\let\OLDthebibliography\thebibliography
\renewcommand\thebibliography[1]{
	\OLDthebibliography{#1}
	\setlength{\parskip}{0pt}
	\setlength{\itemsep}{0pt plus 0.3ex}
}

\title{Entropy and polynomial entropy of derived autoequivalences of derived discrete algebras}
\author{Tomasz Ciborski}
\address{Tomasz Ciborski\newline
	Faculty of Mathematics and Computer Science\newline
	Nicolaus Copernicus University\newline
	ul. Chopina 12/18\newline
	87-100 Toru\'n\newline
	Poland}
\email{tomaszcib@mat.umk.pl}



\begin{document}
\maketitle
\begin{abstract}
	The aim of this paper is to calculate entropy in the sense of Dimitrov--Haiden--Katzarkov--Kontsevich and polynomial entropy as defined by Fan--Fu--Ouchi of derived autoequivalences of derived discrete algebras over an algebraically closed field.
\end{abstract}
\section{Introduction}
The notion of entropy of exact endofunctors of triangulated and $A_\infty$-categories originates from the study of the entropy of dynamical systems also known as the topological entropy, where it is used to measure (with some precision), the exponential growth rate of a certain number associated with a given continuous map on a compact metric space (see, for example \cite{katok}).

The idea to look at the entropy in the context of the triangulated category theory has been introduced by Dimitrov, Haiden, Katzarkov and Kontsevich in \cite{haiden}. Having a generator and an endofunctor of a triangulated category, we define the entropy of the endofunctor as a measure of the exponential growth of `distance in exact triangles' between the generator and its images under consecutive powers of the endofunctor. It is also shown in \cite{haiden} that in the case of the saturated $A_\infty$-categories, the entropy of a given endofunctor can be expressed as a limit of Poincar\'e polynomials of $\Ext$-groups. Additionally, Fan, Fu and Ouchi introduced in \cite{fan-fu-ouchi} a notion of a polynomial entropy, which is meant to measure the polynomial growth of the `distance in exact triangles' under a given endofunctor of a triangulated category. Both works come with some examples of calculations of the entropy and the polynomial entropy of various functors in a number of cases.

Throughout this paper we work over an algebraically closed field $\mathbb{K}$ and we assume that all functors and categories are $\mathbb{K}$-linear. We also assume all functors between triangulated categories are exact.

The aim of the paper is to calculate the entropy and the polynomial entropy of derived autoequivalences of derived discrete algebras. The classification of the derived discrete algebras up to derived equivalences has been given in \cite{bobinski}. Namely, if $\Lambda$ is a derived discrete $\mathbb{K}$-algebra, which is not piecewise hereditary of Dynkin type, then $\Lambda$ is derived equivalent to the algebra $\Lambda(p,q,r)$ (see Figure \ref{quiver-figure} in Subsection \ref{section-25}), where $q\geq 0$ and $1\leq r\leq p$ are uniquely determined up to derived equivalence and $p=r$ if and only if $\gldim\Lambda=\infty$. It is worth mentioning that the parameters $p,q,r$ can be calculated combinatorically using the Avella-Alaminos--Geiss invariant (see \cite{avella-alaminos-geiss}).

The first main result is as follows (for the definitions of $T_{\mathbf{X},\Lambda}$ and $T_{\mathbf{Y},\Lambda}$ we refer to Subsections~\ref{section-24} and~\ref{section-25}):
\begin{theorem}\label{theorem-finite-gldim}
	Let $\Lambda$ be a derived discrete $\mathbb{K}$-algebra, which is not piecewise hereditary of Dynkin type and of finite global dimension. Suppose that $\Lambda$ is derived equivalent to $\Lambda(p,q,r)$ where $q\geq 0$ and $1\leq r<p$. If $F$ is an autoequivalence of $\mathbf{D}^\mathrm{b}(\mod\Lambda)$, then there exist $k,l,s\in\mathbb{Z}$ such that $F(A)\cong T_{\mathbf{X},\Lambda}^kT_{\mathbf{Y},\Lambda}^l\Sigma^s(A)$ for all $A\in\mathbf{D}^\mathrm{b}(\mod\Lambda)$ and the following hold:
	\begin{enumerate}[noitemsep,label=(\arabic*)]
		\item The entropy function $h_F\colon\mathbb{R}\rightarrow\mathbb{R}$ of $F$ is given by \[h_F(t)=st+\max\left(l\frac{r}{r-p}t, k\frac{r}{r+q}t\right).\]
		\item The polynomial entropy function $\hat{h}_F\colon\mathbb{R}\rightarrow\mathbb{R}$ of $F$ is given by \[\hat{h}_F(t)=\begin{cases*}
				      1            & if $l(r+q)\neq k(r-p)$ and $t=0$, \\
				      0 & otherwise.
			      \end{cases*}\]
	\end{enumerate}
\end{theorem}

We need to mention that the description of the action of the derived equivalences of derived discrete algebras of finite global dimension on the isomorphism classes of objects of the bounded derived category follows from the description of the groups of the derived autoequivalences of the algebras $\Lambda(p,q,r)$ given in \cite{broomhead}.

We also provide a result in the case of derived discrete algebras of infinite global dimension, for autoequivalences for both their derived categories and the categories of perfect complexes:
\begin{theorem}\label{theorem-infinite-gldim}
	Let $\Lambda$ be a derived discrete $\mathbb{K}$-algebra of infinite global dimension. Suppose that $\Lambda$ is derived equivalent to $\Lambda(p,q,p)$ where $q\geq 0$ and $1\leq p$. If $F$ is an autoequivalence of $\mathcal{T}\in\{\mathbf{D}^\mathrm{b}(\mod\Lambda), \per(\Lambda)\}$, then there exist $k,s\in\mathbb{Z}$ such that $F(A)\cong T_{\mathbf{X},\Lambda}^k\Sigma^s(A)$ for all $A\in\mathcal{T}$ and the following hold:
	\begin{enumerate}[noitemsep,label=(\arabic*)]
		\item The entropy function $h_F\colon\mathbb{R}\rightarrow\mathbb{R}$ of $F$ is given by \[h_F(t)=st+k\frac{p}{p+q}t.\]
		\item The polynomial entropy function $\hat{h}_F\colon\mathbb{R}\rightarrow\mathbb{R}$ of $F$ is given by \[\hat{h}_F(t)=0.\]
	\end{enumerate}
\end{theorem}

It can be further noted, that by work of Ikeda in \cite{ikeda}, the entropy of an endofunctor of a triangulated category is closely related to the notion of the mass growth of an endofunctor. Namely, by \cite[Proposition 4.1]{ikeda} we know that for a derived autoequivalence $F$ of a derived discrete algebra $\Lambda$ and given $t\in\mathbb{R}$, the entropy $h_F(t)$ coincides with the mass growth in $t$ of $F$ with respect to a stability condition corresponding to the standard heart. It is also shown in \cite[Lemma 3.7]{fan-fu-ouchi} that the polynomial entropy $\hat{h}_F(t)$ equals the polynomial mass growth rate in $t$ of $F$ with respect to a stability condition corresponding to the standard heart.

The paper is structured as follows: in Subsections~\ref{section-21} and~\ref{section-22} we provide the definition of the entropy and the categorical entropy and state their basic properties, while in Subsection~\ref{section-shifts} we provide some elementary examples in the case of finite-dimensional algebras. In Subsection~\ref{section-24} we generalize notion of twist functors on triangulated categories and in Subsection~\ref{section-25} we describe the derived autoequivalences of derived discrete algebras in terms of generators. We then proceed to prove Theorem~\ref{theorem-finite-gldim} in Section~\ref{section-3} and Theorem~\ref{theorem-infinite-gldim} in Section~\ref{section-4}.

The author gratefully acknowledges the support of the National Science Centre grant \\no. 2020/37/B/ST1/00127.

\section{Preliminaries}
\subsection{Categorical complexity and entropy}\label{section-21}
Throughout this subsection $\mathcal{T}$ is a fixed triangulated category. A triangulated subcategory of $\mathcal{T}$ is called \emph{thick} if it contains all direct summands of its objects. For an object $M\in\mathcal{T}$ we denote by $\langle M\rangle$ the smallest thick subcategory containing $M$. An object $G\in\mathcal{T}$ is called a \emph{split-generator}  of $\mathcal{T}$ if $\langle G\rangle = \mathcal{T}$. A triangulated category $\mathcal{T}$ is said to be of \emph{finite type} if for all $X,Y\in\mathcal{T}$ we have $\sum_{n\in\mathbb{Z}}\dim_{\mathbb{K}}\Hom_\mathcal{T}(X,\Sigma^n Y)<\infty$. We will say that $\mathcal{T}$ is \emph{saturated} if it is equivalent to the homotopy category of a triangulated saturated $A_\infty$-category.

If $F_1$ and $F_2$ are two endofunctors of $\mathcal{T}$, then we will write $F_1\equiv F_2$ whenever $F_1(X)\cong F_2(X)$ for all objects $X\in\mathcal{T}$.

A sequence of morphisms $A_0\xrightarrow{f_1} A_1\xrightarrow{f_2}\ldots\xrightarrow{f_{n}} A_n$ in $\mathcal{T}$ gives a rise to a sequence of exact triangles $(A_{i-1}\rightarrow A_{i}\rightarrow C_{i}\rightarrow \Sigma A_{i-1})_{i=1,\ldots, n}$, where $C_i\cong\cone(f_i)$. We will call such a sequence a \emph{tower} (if $A_0=0$, then the sequence is usually called a \emph{filtration} of $A_n$) and we will denote it by the following notation:
\begin{center}
	\begin{tikzcd}[row sep=small, column sep=small]
		A_0\arrow[rr]&&A_1\arrow[ld]\arrow[r]&\cdots\arrow[r]&A_{n-1}\arrow[ld]\arrow[rr]&&A_n\arrow[ld]\\
		&C_1\arrow[ul,dashed]&&\arrow[ul,dashed]\cdots&&C_{n}\arrow[ul,dashed]
	\end{tikzcd}
\end{center}
We introduce the following function:
\begin{definition}[{{\cite[Definition 2.1]{haiden}}}]
	For $X,Y\in\mathcal{T}$, we define a function $\delta_{X,Y}\colon \mathbb{R}\rightarrow \mathbb{R}_{\geq 0}\cup\{\infty\}$ by
	\[
		\delta_{X,Y}(t)\coloneqq\inf\left\{ \sum_{i=1}^{k}\exp(n_it)\;|\; \begin{tikzcd}[row sep=tiny, column sep=tiny]
			0\arrow[rr]&&A_1\arrow[ld]\arrow[r]&\cdots\arrow[r]&A_{k-1}\arrow[ld]\arrow[rr]&&Y\oplus B\arrow[ld]\\
			&\Sigma^{n_1}X\arrow[ul,dashed]&&\arrow[ul,dashed]\cdots&&\Sigma^{n_k}X\arrow[ul,dashed]
		\end{tikzcd}\right\}.
	\]
\end{definition}
Note that $\delta_{X,Y}(t)=\infty$ if and only if $Y\notin\langle X\rangle$. The function $\delta_{X,Y}$ is called the \emph{complexity} of $Y$ with respect to $X$.

We recall some basic properties of the complexity function (\cite[Proposition 2.3]{haiden} and \cite[Lemma 2.3]{kikuta}):
\begin{lemma}\label{delta-properties}
	For any $X, Y, Z \in\mathcal{T}$ and any functor $F\colon\mathcal{T}\rightarrow\mathcal{T}'$ between triangulated categories, the following properties hold for all $t\in\mathbb{R}$:
	\begin{enumerate}[noitemsep, label=(\arabic*)]
		\item If $Y\cong Z$, then $\delta_{X,Y}(t)=\delta_{X,Z}(t)$.
		\item $\delta_{X, Z}(t) \leq \delta_{X,Y}(t)\delta_{Y,Z}(t)$.
		\item $\delta_{X, Y\oplus Z}(t) \leq \delta_{X,Y}(t)+\delta_{X,Z}(t)$.
		\item $\delta_{F(X), F(Y)}(t)\leq \delta_{X, Y}(t)$.\qed
	\end{enumerate}
\end{lemma}

By Lemma~\ref{delta-properties}(4) we immediately obtain the following fact:
\begin{corollary}\label{delta-properties-equal-on-equivalence}
	For any $X, Y\in\mathcal{T}$ and any equivalence $F\colon\mathcal{T}\rightarrow\mathcal{T}'$ of triangulated categories we have $\delta_{F(X), F(Y)}(t)=\delta_{X, Y}(t)$ for all $t\in\mathbb{R}$.\qed
\end{corollary}

We may now introduce the following function:
\begin{definition}[{{\cite[Definition 2.5]{haiden}}}]
	Let $G$ be a split-generator of $\mathcal{T}$ and $F$ an endofunctor of $\mathcal{T}$. The \emph{entropy} of $F$ is defined to be the function $h_F\colon \mathbb{R}\rightarrow \mathbb{R}\cup\{-\infty\}$ given by the formula
	\[
		h_F(t)\coloneqq\lim_{n\to\infty}\frac{1}{n}\log(\delta_{G,F^n(G)}(t)).
	\]
\end{definition}
By \cite[Lemma 2.6]{haiden} the entropy of $F$ is well-defined and does not depend on the choice of a split-generator of $\mathcal{T}$. We now list some properties of the entropy function (see remarks after \cite[Lemma 2.6]{haiden}):
\begin{lemma}\label{entropy-properties}
	Let $G$ be a split-generator of $\mathcal{T}$ and $F_1, F_2$ endofunctors of $\mathcal{T}$. Then for all $t\in\mathbb{R}$ we have:
	\begin{enumerate}[noitemsep, label=(\arabic*)]
		\item If $F_1\equiv F_2$, then $h_{F_1}(t)=h_{F_2}(t)$.
		\item $h_{F_1^m}(t) = mh_{F_1}(t)$ for any $m\geq 1$.
		\item If $F_1F_2\cong F_2F_1$, then $h_{F_1F_2}(t)\leq h_{F_1}(t)+h_{F_2}(t)$.\qed
	\end{enumerate}
\end{lemma}

Let $X,Y\in\mathcal{T}$. We define the \emph{Ext-distance} function $\delta_{X, Y}'\colon\mathbb{R}\rightarrow \mathbb{R}\cup\{\infty\}$ by the formula
\[
	\delta'_{X,Y}(t)\coloneqq \sum_{a\in\mathbb{Z}}\dim_\mathbb{K}\Hom_\mathcal{T}(X, \Sigma^aY)\cdot \exp(-at).
\]
Observe that if $\mathcal{T}$ is of finite type, then $\delta_{X, Y}'(t)<\infty$ for any $t\in\mathbb{R}$.

\begin{lemma}\label{entropy-of-adjunction}
	Let $G$ be a split-generator of $\mathcal{T}$, $F\colon\mathcal{T}\rightarrow\mathcal{T}$ a functor and $H\colon\mathcal{T}\rightarrow\mathcal{T}'$ an equivalence between triangulated categories. Then $h_F(t)=h_{HFH^{-1}}(t)$ for all $t\in\mathbb{R}$.
\end{lemma}
\begin{proof}
	Put $F_1\coloneqq HFH^{-1}$. Observe that by Corollary \ref{delta-properties-equal-on-equivalence} we have
	\[
	\delta_{G, F^n(G)}(t)=\delta_{H(G), H(F^n(G))}(t)=\delta_{H(G), F_1^n(H(G))}(t)
	\]
	for all $n\geq 1$ and $t\in\mathbb{R}$. Since $H$ is an equivalence, $H(G)\in\mathcal{T}'$ is a split-generator of $\mathcal{T}'$. Thus,
	\[
	h_F(t)=\lim_{n\to\infty}\frac{1}{n}\log(\delta_{G, F^n(G)}(t))=\lim_{n\to\infty}\frac{1}{n}\log(\delta_{H(G), F_1^n(H(G))}(t))=h_{HFH^{-1}}(t).\qedhere
	\]
\end{proof}

By the proof of \cite[Theorem 2.7]{haiden} the Ext-distance function can, in certain situations, be approximated by the complexity function.
\begin{lemma}\label{ext-distance-approximation-lemma}
	Assume that $\mathcal{T}$ is saturated and $G$ is a split-generator of $\mathcal{T}$. Then there exist functions $C_1, C_2\colon\mathbb{R}\rightarrow\mathbb{R}_{>0}$, depending on $G$, such that
	\[
		C_1(t)\cdot\delta_{G,X}(t)\leq \delta'_{G,X}(t)\leq C_2(t)\cdot\delta_{G,X}(t)
	\]
	for any object $X\in\mathcal{T}$ and all $t\in\mathbb{R}$.\qed
\end{lemma}
This holds in particular for $\mathcal{T}=\per(\Lambda)$ where $\Lambda$ is  a smooth proper dg-algebra over $\mathbb{K}$. Examples of such dg-algebras include finite-dimensional $\mathbb{K}$-algebras of finite global dimension. In this case we have $\per(\Lambda)=\mathbf{D}^\mathrm{b}(\mod\Lambda)$. Lemma~\ref{ext-distance-approximation-lemma} gives us an immediate consequence:
\begin{lemma}\label{entropy-formula}
	Let $\mathcal{T}$ be saturated, $G$ a split-generator of $\mathcal{T}$ and $F$ an endofunctor of $\mathcal{T}$. Then
	\[
		h_F(t)=\lim_{n\to\infty}\frac{1}{n}\delta_{G,F^n(G)}'(t)
	\]
	for all $t\in\mathbb{R}$.\qed
\end{lemma}

We say that an autoequivalence $F$ is \emph{entropy inverse-regular} if for all $t\in\mathbb{R}$ the equality $h_{F^{-1}}(t)=h_F(-t)$ holds. From the proof of \cite[Lemma 2.12(i)]{fan-fu-ouchi} we can deduce the following fact:
\begin{lemma}\label{odd-funct-additive-ent}
	Suppose that $F_1$ is an entropy inverse-regular autoequivalence of $\mathcal{T}$ for which the entropy function $h_{F_1}$ is an odd function. Let $F_2$ be an endofunctor of $\mathcal{T}$. If $F_1F_2\cong F_2F_1$ then the inequality in Lemma~\ref{entropy-properties}(3) becomes an equality, i.e.
	\[
		h_{F_1F_2}(t)=h_{F_1}(t)+h_{F_2}(t)
	\]
	for all $t\in\mathbb{R}$.
\end{lemma}
\begin{proof}
	It follows from Lemma \ref{entropy-properties}(3) that $h_{F_1F_2}(t)\leq h_{F_1}(t)+h_{F_2}(t)$ for all $t\in\mathbb{R}$. On the other hand, by Lemma \ref{entropy-properties}(1),(3) we have 
	\[
	h_{F_1}(t)+h_{F_2}(t) = h_{F_1}(t)+h_{F_1^{-1}F_1F_2}(t)\leq h_{F_1}(t)+h_{F_1^{-1}}(t)+h_{F_1F_2}(t)=h_{F_1F_2}(t)
	\]
	for all $t\in\mathbb{R}$, where the last equality follows from our assumptions that $F_1$ is entropy inverse-regular and $h_{F_1}$ is an odd function.
\end{proof}

We have the following fact (see \cite[Lemma 2.11]{fan-fu-ouchi}):
\begin{lemma}\label{all-funcs-are-inverse-regular}
	Let $\mathcal{T}$ be saturated, $F$ an autoequivalence of $\mathcal{T}$ and assume that $\mathcal{T}$ has a Serre functor. Then $F$ is entropy inverse-regular.\qed
\end{lemma}

\subsection{Polynomial entropy}\label{section-22}
Similarly to the previous subsection, $\mathcal{T}$ will denote a fixed triangulated category.
\begin{definition}[{\cite[Definition 2.4]{fan-fu-ouchi}}]
	Let $G$ be a split-generator of $\mathcal{T}$ and $F$ an endofunctor of $\mathcal{T}$. The \emph{polynomial entropy} of $F$ is defined to be the function $\hat{h}_F\colon \{t\in\mathbb{R}\;|\; h_F(t)\neq-\infty\}\rightarrow\mathbb{R}\cup\{\pm\infty\}$ given by the formula
	\[
		\hat{h}_F(t)\coloneqq \limsup_{n\to\infty}\frac{\log(\delta_{G, F^n(G)}(t))-nh_F(t)}{\log(n)}.
	\]
\end{definition}
The polynomial entropy $\hat{h}_F(t)$ is well-defined for any $t\in\mathbb{R}$ such that $h_F(t)\neq-\infty$ and does not depend on the choice of a split-generator \cite[Definition 2.4 and Lemma 2.6]{fan-fu-ouchi}.

We now list some properties of the polynomial entropy. We will start with the following obvious fact:
\begin{lemma}\label{poly-entropy-on-equivalent-functors}
	Let $G$ be a split-generator of $\mathcal{T}$ and $F$ and $F_1 $ endofunctors of $\mathcal{T}$ such that $F\equiv F_1$. Then $\hat{h}_F(t)=\hat{h}_{F_1}(t)$ for all $t\in\mathbb{R}$ such that $h_F(t)\neq-\infty$.\qed
\end{lemma}

By Lemma \ref{entropy-of-adjunction} and arguments used in its proof, we can state the following:
\begin{lemma}\label{polynomial-entropy-of-adjunction}
	Let $G$ be a split-generator of $\mathcal{T}$, $F\colon\mathcal{T}\rightarrow\mathcal{T}$ an endofunctor and $H\colon\mathcal{T}\rightarrow\mathcal{T}'$ an equivalence between triangulated categories. Then $\hat{h}_F(t)=\hat{h}_{HFH^{-1}}(t)$ for all $t\in\mathbb{R}$ such that $h_F(t)\neq-\infty$.\qed
\end{lemma}

We also have the following:

\begin{lemma}\label{poly-entropy-formula}
	Assume that $\mathcal{T}$ is saturated. Let $G$ be a split-generator of $\mathcal{T}$ and $F$ an endofunctor of $\mathcal{T}$. Then
	\begin{enumerate}[noitemsep, label=(\arabic*)]
		\item $\hat{h}_F(t) = \limsup_{n\to\infty}\frac{\log(\delta'_{G, F^n(G)}(t))-nh_F(t)}{\log(n)}$,
		\item $\hat{h}_{F}(t)=\lim_{n\to\infty}\frac{\log(\delta'_{G, F^n(G)}(t))-nh_{F}(t)}{\log(n)}$ if and only if $\hat{h}_{F}(t)=\lim_{n\to\infty}\frac{\log(\delta_{G, F^n(G)}(t))-nh_{F}(t)}{\log(n)}$,
	\end{enumerate}
	for all $t\in\mathbb{R}$ such that $h_F(t)\neq-\infty$.
\end{lemma}
\begin{proof}
	(1) is stated in \cite[Lemma 2.7]{fan-fu-ouchi}. (2) can be easily shown using Lemma~\ref{ext-distance-approximation-lemma}.
\end{proof}
\begin{lemma}\label{cat-poly-powers}
	Let $G$ be a split-generator of $\mathcal{T}$ and $F$ an endofunctor of $\mathcal{T}$. If $t\in\mathbb{R}$ and
	\[
		\hat{h}_{F}(t)=\lim_{n\to\infty}\frac{\log(\delta_{G, F^n(G)}(t))-nh_F(t)}{\log(n)},
	\]
	then $\hat{h}_{F^m}(t)=\hat{h}_{F}(t)$ for all $m>0$. Moreover, in the above situation
	\[
		\hat{h}_{F^m}(t)=\lim_{n\to\infty}\frac{\log(\delta_{G, F^{nm}(G)}(t))-nh_{F^m}(t)}{\log(n)}.
	\]
\end{lemma}
\begin{proof}
	The first part of this lemma is stated in \cite[Lemma 2.10]{fan-fu-ouchi}. If $f(n)\coloneqq \frac{\log(\delta_{G, F^n(G)}(t))-nh_F(t)}{\log(n)}$, then
	\begin{equation*}
		\begin{split}
			\hat{h}_{F^m}(t)=\hat{h}_{F}(t)&=\lim_{n\to\infty}f(n)= \lim_{n\to\infty} f(nm) = \lim_{n\to\infty}\frac{\log(\delta_{G, F^{nm}(G)}(t))-nmh_F(t)}{\log(nm)}\\&=\lim_{n\to\infty}\frac{\log(\delta_{G, F^{nm}(G)}(t))-nh_{F^m}(t)}{\log(n)},
		\end{split}
	\end{equation*}
	where Lemma~\ref{entropy-properties}(2) is used to show the last equality.
\end{proof}

If the polynomial entropy is a limit with respect to a split-generator and the Ext-distance function, it remains a limit if we substitute the second object in the Ext-distance function with another split-generator.
\begin{lemma}\label{poly-ent-independent-on-gens-prime}
	Let $G$ and $G'$ be split-generators of $\mathcal{T}$ and $F$ an endofunctor of $\mathcal{T}$. Assume that $\mathcal{T}$ is saturated and that $\delta_{G,G}(t)>0$ for all $t\in\mathbb{R}$. If
	\[
		\hat{h}_F(t)=\lim_{n\to\infty}\frac{\log(\delta'_{G, F^n(G)}(t))-nh_F(t)}{\log(n)},
	\]
	then
	\[
		\hat{h}_F(t)=\lim_{n\to\infty}\frac{\log(\delta'_{G, F^n(G')}(t))-nh_F(t)}{\log(n)}
	\]
	for all $t\in\mathbb{R}$ where $h_F(t)\neq-\infty$.
\end{lemma}
\begin{proof}
	Observe first, that by Lemma~\ref{delta-properties}(2) we have $0<\delta_{G,G}(t)\leq \delta_{G,G'}(t)\cdot\delta_{G',G}(t)$, hence $\delta_{G,G'}(t),\delta_{G',G}(t)>0$.

	By Lemma~\ref{ext-distance-approximation-lemma} there exist functions $C_1',C_1''\colon\mathbb{R}\rightarrow\mathbb{R}_{>0}$ such that $\delta'_{G,F^n(G)}(t)\leq C_1'(t)\cdot~\delta_{G,F^n(G)}(t)$ and $\delta_{G,F^n(G')}(t)\leq C_1''(t)\cdot \delta'_{G,F^n(G')}(t)$ for all $t\in\mathbb{R}$ and $n\geq 0$. We have the following chain of inequalities by Lemma~\ref{delta-properties}(2),(4)
	\begin{equation*}
		\begin{split}
			\delta_{G,F^n(G)}'(t)&\leq C_1'(t)\cdot \delta_{G,F^n(G')}(t)\cdot \delta_{F^n(G'),F^n(G)}(t)\leq C_1'(t)\cdot \delta_{G,F^n(G')}(t)\cdot \delta_{G',G}(t)\\
			&\leq C_1(t)\cdot\delta'_{G,F^n(G')}(t),
		\end{split}
	\end{equation*}
	where $C_1(t)\coloneqq C_1'(t)\cdot\delta_{G',G}(t)\cdot C_1''(t)>0$.

	On the other hand, again by Lemma~\ref{ext-distance-approximation-lemma} there exist functions $C_2',C_2''\colon\mathbb{R}\rightarrow\mathbb{R}_{>0}$ such that $\delta'_{G,F^n(G')}(t)\leq C_2'(t)\cdot\delta_{G,F^n(G')}(t)$ and $\delta_{G,F^n(G)}(t)\leq C_2''(t)\cdot \delta'_{G,F^n(G)}(t)$. Then, we have a chain of inequalities
	\begin{equation*}
		\begin{split}
			\delta'_{G,F^n(G')}(t)&\leq C_2'(t)\cdot\delta_{G,F^n(G)}(t)\cdot\delta_{F^n(G),F^n(G')}(t)\leq C_2'(t)\cdot\delta_{G,F^n(G)}(t)\cdot\delta_{G,G'}(t)\\
			&\leq C_2(t)\cdot \delta'_{G,F^n(G)}(t),
		\end{split}
	\end{equation*}
	where $C_2(t)\coloneqq C_2'(t)\cdot \delta_{G,G'}(t)\cdot C_2''(t)>0$. Consequently, we have the following inequalities
	\begin{equation*}
		\begin{split}
			\frac{\log(\frac{1}{C_1(t)}\cdot\delta'_{G, F^n(G)}(t))-nh_F(t)}{\log(n)}\leq \frac{\log(\delta'_{G, F^n(G')}(t))-nh_F(t)}{\log(n)}\leq \frac{\log(C_2(t)\cdot \delta'_{G, F^n(G)}(t))-nh_F(t)}{\log(n)}
		\end{split}
	\end{equation*}
	to which we apply the squeeze theorem and obtain the desired result.
\end{proof}

\begin{lemma}\label{cat-poly-powers-inverse}
	Let $F$ be an autoequivalence of $\mathcal{T}$ and $G$ a split-generator of $\mathcal{T}$. Fix $t\in\mathbb{R}$ such that $h_F(t)\neq-\infty$, $m>0$ and assume that \[\hat{h}_{F^{m}}(t) = \lim_{n\to\infty}\frac{\log(\delta_{G, F^{nm}(G)}(t))-nh_{F^m}(t)}{\log(n)}\]
	and $\delta_{G,F^{k}(G)}(t)>0$ for all $k\in\{-m+1,\ldots, 0,\ldots, m-1\}$. Then $\hat{h}_{F}(t)=\hat{h}_{F^m}(t)$.
\end{lemma}
\begin{proof}
	Fix $k\in\{0,\ldots, m-1\}$. We obtain the following chain of inequalities
	\begin{equation*}
		\begin{split}
			\delta_{G,F^{nm+k}(G)}(t)&\leq\delta_{G,F^k(G)}(t)\cdot \delta_{F^k(G), F^{nm+k}(G)}(t)
			\\& = \delta_{G,F^k(G)}(t)\cdot \delta_{G, F^{nm}(G)}(t)= C_{2}(t)\cdot \delta_{G,F^{nm}(G)}(t),
		\end{split}
	\end{equation*}
	where $C_2(t)\coloneqq \delta_{G,F^k(G)}(t)$. The first inequality follows from Lemma~\ref{delta-properties}(2), and the next equality from Corollary~\ref{delta-properties-equal-on-equivalence}. On the other hand, we have the following chain of inequalities
	\begin{equation*}
		\begin{split}
			\delta_{G, F^{nm}(G)}(t)\leq \delta_{G,F^{-k}(G)}(t)\cdot \delta_{F^{-k}(G), F^{nm}(G)}(t)= \delta_{G,F^{-k}(G)}(t)\cdot \delta_{G, F^{nm+k}(G)}(t).
		\end{split}
	\end{equation*}
	Again, the first inequality follows from Lemma~\ref{delta-properties}(2) and the next equality from Corollary~\ref{delta-properties-equal-on-equivalence}. We now multiply the above chain of inequalities by $C_1(t)\coloneqq\frac{1}{\delta_{G,F^{-k}(G)}(t)}$ and infer that we have the following two inequalities
	\begin{equation}\label{squeeze-limit}
		C_1(t)\cdot \delta_{G,F^{nm}(G)}(t) \leq \delta_{G, F^{nm+k}(G)}(t) \leq C_2(t)\cdot \delta_{G,F^{nm}(G)}(t).
	\end{equation}
	For $i=1,2$ we easily obtain
	\begin{equation*}
		\lim_{n\to\infty}\frac{\log(C_i(t)\cdot \delta_{G, F^{nm}(G)}(t))-nmh_F(t)}{\log(n)} = \lim_{n\to\infty}\frac{\log(\delta_{G, F^{nm}(G)}(t))-nh_{F^{m}}(t)}{\log(n)}=\hat{h}_{F^{m}}(t),
	\end{equation*}
	using Lemma~\ref{entropy-properties}(2) and our assumption. From inequalities (\ref{squeeze-limit}) and by applying the squeeze theorem we get
	\begin{equation*}
		\begin{split}
			\hat{h}_{F^{m}}(t)&=\lim_{n\to\infty}\frac{\log(\delta_{G, F^{nm+k}(G)}(t))-nmh_F(t)}{\log(n)}
			=\lim_{n\to\infty}\frac{\log(\delta_{G, F^{nm+k}(G)}(t))-(nm+k)h_F(t)}{\log(n)}\\
			&=\lim_{n\to\infty}\frac{\log(\delta_{G, F^{nm+k}(G)}(t))-(nm+k)h_F(t)}{\log(nm+k)}.
		\end{split}
	\end{equation*}
	By putting $f(n)\coloneqq \frac{\delta_{G,F^n(G)}(t)-nh_F(t)}{\log(n)}$, we infer that $\lim_{n\to\infty}f(nm+k)=\hat{h}_{F^m}(t)$. Since $k\in\{0,\ldots, m-1\}$ was arbitrary, we deduce that $\lim_{n\to\infty}f(n)=\hat{h}_{F^m}(t)$, which is just equivalent to $\hat{h}_{F}(t)=\hat{h}_{F^m}(t)$.
\end{proof}

\begin{lemma}\label{poly-ent-properties}
	Let $G$ be a split-generator of $\mathcal{T}$ and $F$ an autoequivalence on $\mathcal{T}$. Then:
	\begin{enumerate}[noitemsep, label=(\arabic*)]
		\item Assume that $\mathcal{T}$ is saturated and has a Serre functor $S$. Then $\hat{h}_{F^{-1}}(t)=\hat{h}_F(-t)$ for all $t\in\mathbb{R}$ such that $h_F(t)\neq-\infty$. Moreover, if for all $t\in\mathbb{R}$ such that $h_F(t)\neq-\infty$ we have
		      \[
			      \hat{h}_{F}(t)=\lim_{n\to\infty}\frac{\log(\delta'_{G, F^n(G)}(t))-nh_F(t)}{\log(n)}
		      \]
		      then for all $t\in\mathbb{R}$ such that $h_F(t)\neq-\infty$ we also have
		      \[
			      \hat{h}_{F^{-1}}(t)=\lim_{n\to\infty}\frac{\log(\delta'_{G,F^{-n}(G)}(t))-nh_{F^{-1}}(t)}{\log(n)}.
		      \]
		\item Suppose that $F$ is entropy inverse-regular, $h_{F}$ is an odd function and $F_1$ is an endofunctor of $\mathcal{T}$ such that $FF_1\cong F_1F$. Then $\hat{h}_{F_1F}(t)\leq \hat{h}_{F_1}(t)+\hat{h}_{F}(t)$ for all $t\in\mathbb{R}$ such that $h_F(t)\neq-\infty$. Moreover, if $\hat{h}_{F}$ is an odd function, then we have an equality.
	\end{enumerate}
\end{lemma}
\begin{proof}
	(1) The first part is stated in \cite[Lemma 2.11]{fan-fu-ouchi}. For the second part, observe that
	\begin{equation*}
		\begin{split}
			\hat{h}_F(-t)&=\lim_{n\to\infty}\frac{\log(\delta'_{G,F^{n}(G)}(-t))-nh_F(-t)}{\log(n)}=\lim_{n\to\infty}\frac{\log(\delta'_{G,F^{n}(S(G))}(-t))-nh_F(-t)}{\log(n)}\\
			&=\lim_{n\to\infty}\frac{\log(\delta'_{G,F^{-n}(G)}(t))-nh_{F^{-1}}(t)}{\log(n)}=\hat{h}_{F^{-1}}(t).
		\end{split}
	\end{equation*}
	The second equality follows from Lemma~\ref{poly-ent-independent-on-gens-prime} and the third one from properties of the Serre functor and the fact that $F$ is entropy inverse-regular, which is guaranteed by Lemma~\ref{all-funcs-are-inverse-regular}.

	(2) This follows from the proof of \cite[Lemma 2.12]{fan-fu-ouchi}.
\end{proof}

\subsection{Example: shift functors}\label{section-shifts}
Throughout this subsection we will denote by $\Lambda$ a finite-dimensional $\mathbb{K}$-algebra and by $G$ a split-generator of the bounded derived category $\mathbf{D}^\mathrm{b}(\mod\Lambda)$ which is a stalk complex concentrated in degree zero. As usual, $\Sigma$ will denote the shift functor on $\mathbf{D}^\mathrm{b}(\mod\Lambda)$. For $a,b\in\mathbb{Z}$, by $\mathbf{D}^{[a,b]}(\mod\Lambda)$ we will understand the subcategory of $\mathbf{D}^\mathrm{b}(\mod\Lambda)$ consisting of the objects $A\in\mathbf{D}^\mathrm{b}(\mod\Lambda)$ for which $H^{k}(A)=0$ for all $k\in\mathbb{Z}$ such that $k<a$ or $k>b$.

We will start with stating a technical lemma:
\begin{lemma}\label{triangle-tower-homologies}
	Assume that $a\leq b$.
	\begin{enumerate}[noitemsep, label=(\arabic*)]
		\item 	Let
		      \begin{center}
			      \begin{tikzcd}[row sep=small, column sep=small]
				      A_0\arrow[rr]&&A_1\arrow[ld]\arrow[r]&\cdots\arrow[r]&A_{n-1}\arrow[ld]\arrow[rr]&&A_n\arrow[ld]\\
				      &C_1\arrow[ul,dashed]&&\arrow[ul,dashed]\cdots&&C_{n}\arrow[ul,dashed]
			      \end{tikzcd}
		      \end{center}
		      be a tower of exact triangles in $\mathbf{D}^\mathrm{b}(\mod\Lambda)$ such that $A_0, C_1,\ldots, C_{n}\in \mathbf{D}^{[a,b]}(\mod\Lambda)$. Then $A_1,\ldots, A_n\in\mathbf{D}^{[a,b]}(\mod\Lambda)$.
		\item Let $A\xrightarrow{f} B\rightarrow C\rightarrow \Sigma A$ be an exact triangle, where $A\in\mathbf{D}^{[a,a]}(\mod\Lambda)$ and $B\in\mathbf{D}^{[a,b]}(\mod\Lambda)$. If $H^a(f)\colon H^a(A)\rightarrow H^a(B)$ is a monomorphism, then $C\in\mathbf{D}^{[a,b]}(\mod\Lambda)$.
		\item Let $A\rightarrow B\rightarrow C\rightarrow \Sigma A$ be an exact triangle, where $A\in\mathbf{D}^{[b+k,b+k]}(\mod\Lambda)$ for some $k\geq 1$, $H^{b+k}(A)\neq0$, $B\in\mathbf{D}^{[a,b]}(\mod\Lambda)$ and $H^a(B)\neq 0$. Then $C\in \mathbf{D}^{[a,b+k-1]}(\mod\Lambda)$ and $H^{b+k-1}(C)\neq 0$ and $H^a(C)\neq 0$.
	\end{enumerate}

\end{lemma}
\begin{proof}
	(1) Since $H^{c}(A_0)=H^{c}(C_1)=0$, we infer from the long exact cohomology sequence that $H^{c}(A_1)=0$, for all $c<a$. On the other hand, since $H^{c}(A_0)=H^{c}(C_1)=0$, we infer that $H^{c}(A_1)=0$, for all $c>b$ . We proceed by a simple induction.

	(2), (3) Similarly to (1), these are seen easily by using the long exact sequences of cohomologies.
\end{proof}
This allows us to show the following fact:
\begin{lemma}\label{delta-value-of-sigma}
	For all $k\in\mathbb{Z}$ and $t\in\mathbb{R}$ we have $\delta_{G, \Sigma^kG}(t)=\exp(kt)$.
\end{lemma}
\begin{proof}
	Observe that for any $k$ there exists an exact triangle $0\rightarrow \Sigma^k G\xrightarrow{\cong} \Sigma^kG\rightarrow 0$. By definition, we get  $\delta_{G,\Sigma^kG}(t)\leq \exp(kt)$. Note that for $t=0$ we necessarily have an equality.

	Assume now that there exists $t\neq 0$ such that $\delta_{G, \Sigma^kG}(t)<\exp(kt)$. This implies that there exists a tower of exact triangles
	\begin{center}
		\begin{tikzcd}[row sep=tiny, column sep=tiny]
			0\arrow[rr]&&A_1\arrow[ld]\arrow[r]&\cdots\arrow[r]&\arrow[ld]A_{k-1}\arrow[rr]&&\Sigma^kG\oplus B\arrow[ld]\\
			&\Sigma^{n_1}G\arrow[ul,dashed]&&\arrow[ul,dashed]\cdots&&\Sigma^{n_k}G\arrow[ul,dashed]
		\end{tikzcd}
	\end{center}
	such that $\sum_{i=1}^{k}\exp(n_it)<\exp(kt)$. This means that for all $i=1,\ldots, k$ we get $n_it<kt$, which is equivalent to all $n_i$ being either greater than $k$ if $t<0$ or smaller than $k$ if $t>0$. Observe
	that for all $i=1,\ldots, n$ we have $\Sigma^{n_i}G\in\mathbf{D}^{[\min_j(-n_j), \max_j(-n_j)]}(\mod\Lambda)$. Lemma~\ref{triangle-tower-homologies}(1) implies that $\Sigma^kG\oplus B\in\mathbf{D}^{[\min_j(-n_j), \max_j(-n_j)]}(\mod\Lambda)$, a contradiction since $H^{-k}(\Sigma^kG\oplus B)\neq 0$ and $-k\notin [\min_j(-n_j), \max_j(-n_j)]$.
\end{proof}

We immediately obtain the following result:
\begin{proposition}\label{entropy-of-sigma}
	Let $k\in\mathbb{Z}$. The following equalities hold for all $t\in\mathbb{R}$:
	\begin{enumerate}[noitemsep, label=(\arabic*)]
		\item $h_{\Sigma^k}(t)=kt$,
		\item $\hat{h}_{\Sigma^k}(t)=\lim_{n\to\infty}\frac{\log(\delta_{G,\Sigma^{nk}G}(t))-nh_{\Sigma^k}(t)}{\log(n)}=0$.\qed
	\end{enumerate}
\end{proposition}
\begin{remark}
	Although in this paper we compute $h_{\Sigma^k}(t)$ directly from the definition of the entropy, we can also use \cite[Proposition 4.3]{ikeda} to obtain our result.
\end{remark}

\begin{corollary}\label{sigma-properties}
	If $F$ is an endofunctor of $\mathbf{D}^\mathrm{b}(\mod\Lambda)$, then for any $k\in\mathbb{Z}$ the following equalities hold for all $t\in\mathbb{R}$:
	\begin{enumerate}[noitemsep, label=(\arabic*)]
		\item $h_{\Sigma^kF}(t) = kt+h_F(t)$,
		\item $\hat{h}_{\Sigma^kF}(t) = \hat{h}_F(t)$.
	\end{enumerate}
\end{corollary}
\begin{proof}
	Since $\Sigma$ is entropy-inverse regular and $h_\Sigma$ is an odd function, we may apply Lemma~\ref{odd-funct-additive-ent} to obtain (1). We use Lemma~\ref{poly-ent-properties}(2) to obtain (2).
\end{proof}
For our future calculations, we will also need the following technical result:
\begin{lemma}\label{limit-of-functor-induces-another-limit}
	Let $F$ and $F_1$ be endofunctors on $\mathbf{D}^\mathrm{b}(\mod\Lambda)$. Suppose there exist $k\in\mathbb{Z}$ such that $F\equiv\Sigma^kF_1$. Furthermore, assume that
	\[
		\hat{h}_{F_1}(t)=\lim_{n\to\infty}\frac{\log(\delta_{G, F_1^n(G)}(t))-nh_{F_1}(t)}{\log(n)}.
	\]
	Then
	\[
		\hat{h}_{F}(t)=\lim_{n\to\infty}\frac{\log(\delta_{G, F^n(G)}(t))-nh_{F}(t)}{\log(n)}.
	\]
\end{lemma}
\begin{proof}
	Observe that we have the following chain of inequalities by Lemma~\ref{delta-properties}(2), Corollary~\ref{delta-properties-equal-on-equivalence}, Lemma~\ref{delta-value-of-sigma} and Proposition~\ref{entropy-of-sigma}(1)
	\begin{equation*}
		\begin{split}
			\log(\delta_{G, F_1^n(G)})(t)-nh_{F_1}(t)=&\log(\delta_{G, \Sigma^{-kn}F^n(G)})-nh_{\Sigma^{-k}F}(t)\\&\leq \log(\delta_{G, \Sigma^{-kn}G}(t))+\log(\delta_{\Sigma^{-kn}G, \Sigma^{-kn}F^n(G)}(t))-nh_F(t)+nkt\\
			&= -nkt +\log(\delta_{G, F^n(G)}(t))-nh_F(t)+nkt \\
			&=\log(\delta_{G, F^n(G)}(t))-nh_F(t).
		\end{split}
	\end{equation*}
	On the other hand, using the same properties we get
	\begin{equation*}
		\begin{split}
			\log(\delta_{G, F^n(G)}(t))-nh_F(t)&
			\leq \log(\delta_{G,\Sigma^{nk}G}(t)\cdot\delta_{\Sigma^{nk}G, \Sigma^{nk}F_1^n(G)}(t))-nh_{\Sigma^kF_1}(t)\\
			&= nkt+\log(\delta_{G,F_1^n(G)}(t))-nh_{F_1}(t)-nkt\\
			&=\log(\delta_{G,F_1^n(G)}(t)) - nh_{F_1}(t).
		\end{split}
	\end{equation*}
	By dividing both chains of inequalities by $\log(n)$ and applying the squeeze theorem to these inequalities we obtain the desired result.
\end{proof}

Lemma~\ref{delta-value-of-sigma} allows us to state another useful fact:
\begin{corollary}\label{all-deltas-are-positive}
	For any object $X\in\mathbf{D}^\mathrm{b}(\mod\Lambda)$ and all $t\in\mathbb{R}$ we have $\delta_{G, X}(t)>0$.
\end{corollary}
\begin{proof}
	By Lemma~\ref{delta-properties}(2) we have the following inequality
	\[
		0<\exp(0)=\delta_{G, G}(t)\leq \delta_{G,X}(t)\cdot \delta_{X, G}(t),
	\]
	thus $\delta_{G,X}(t), \delta_{X, G}(t)>0$.
\end{proof}

Note that we can prove similar results for the category $\per(\Lambda)\cong\mathbf{K}^\mathrm{b}(\proj \Lambda)$. As a split-generator we can take $\Lambda$. We have:
\begin{proposition}
	Assume that $\Lambda$ is as above and $\Sigma$ is the shift functor on $\per(\Lambda)$. The following equalities hold for all $t\in\mathbb{R}$:
	\begin{enumerate}[noitemsep, label=(\arabic*)]
		\item $h_{\Sigma^k}(t)=kt$,
		\item $\hat{h}_{\Sigma^k}(t)=\lim_{n\to\infty}\frac{\log(\delta_{\Lambda,\Sigma^{nk}\Lambda}(t))-nh_{\Sigma^k}(t)}{\log(n)}=0$.\qed
	\end{enumerate}
\end{proposition}
\begin{corollary}
	If $F$ is an endofunctor of $\per(\Lambda)$, then for any $k\in\mathbb{Z}$ the following equalities hold for all $t\in\mathbb{R}$:
	\begin{enumerate}[noitemsep, label=(\arabic*)]
		\item $h_{\Sigma^kF}(t) = kt+h_F(t)$,
		\item $\hat{h}_{\Sigma^kF}(t) = \hat{h}_F(t)$.\qed
	\end{enumerate}
\end{corollary}
\begin{corollary}
	For any object $X\in\per(\Lambda)$ and all $t\in\mathbb{R}$ we have $\delta_{\Lambda, X}(t)>0$. \qed
\end{corollary}

\subsection{Twist functors}\label{section-24}

In order to describe the derived autoequivalences of the derived discrete algebras, we need to recall some definitions and properties from \cite{broomhead}, which we will slightly modify. In the original setting, the twist functors are defined for triangulated categories with Serre functors. In this paper we assume that the Serre functors act on sufficiently big subcategories.

\begin{definition}[{\cite[Definition 4.2]{broomhead}}]\label{definition-exceptional-cycle}
	Let $\mathcal{T}$ be an algebraic $\Hom$-finite, indecomposable triangulated category which has a Serre functor $S$. Assume that $n\geq 2$. A sequence $\mathbf{E}=(E_1,\ldots, E_n)$ of objects of $\mathcal{T}$ is called an \emph{exceptional $n$-cycle} if:
	\begin{enumerate}[noitemsep, label=(\arabic*)]
		\item $\Hom^\bullet_\mathcal{T}(E_i, E_i)=\mathbb{K}\cdot\id_{E_i}$ for all $i=1,\ldots, n$,
		\item there are integers $k_i$ such that $S(E_i)\cong \Sigma^{k_i}E_{i+1}$ for all $i$, where $E_{n+1}\coloneqq E_1$,
		\item $\Hom^\bullet_\mathcal{T}(E_i, E_j)=0$ unless $j=i$, $j=i+1$ if $i<n$, or $j=1$ and $i=n$,
	\end{enumerate}
	where $\Hom^\bullet_\mathcal{T}(A,B)\coloneqq \bigoplus_{k\in\mathbb{Z}}\Sigma^k\Hom_\mathcal{T}(\Sigma^kA, B)$ is the \emph{graded homomorphism space}. We will use the notation $k_\mathbf{E}\coloneqq (k_1,\ldots, k_n)$ to denote the sequence of integers mentioned in (2). If $n=1$, then a sequence $\mathbf{E}=(E_1)$ with $E_1\in\mathcal{T}$ is an \emph{exceptional $1$-cycle} if there exists $k_1\in\mathbb{Z}$ such that $S(E_1)\cong \Sigma^{k_1}E_1$ and $\Hom_{\mathcal{T}}^\bullet(E_1,E_1)\cong\mathbb{K}\oplus \Sigma^{-k_1}\mathbb{K}$. In this case we put $k_\mathbf{E}\coloneqq (k_1)$.
\end{definition}
Throughout this section, we will denote by $\mathcal{T}$ an algebraic Hom-finite indecomposable triangulated category, and by $\mathcal{T}'$ a full indecomposable subcategory of $\mathcal{T}$ which has a Serre functor $S$ such that for $A\in\mathcal{T}'$ and $B\in\mathcal{T}$ we have $\Hom_{\mathcal{T}}(A,\Sigma^iB)=0$ for almost all $i\in\mathbb{Z}$.
\begin{definition}\label{twist-functor-definition}
	Assume that $\mathbf{E}=(E_1,\ldots, E_n)$ is an exceptional $n$-cycle in $\mathcal{T}'$. We define an endofunctor $F_\mathbf{E}\colon\mathcal{T}\rightarrow\mathcal{T}$ by the formula
	\[
		F_\mathbf{E}(-)\coloneqq\bigoplus_{i=1}^n\Hom^\bullet_\mathcal{T}(E_i, -)\otimes E_i.
	\]
	Then there exists a canonical evaluation morphism $F_\mathbf{E}\rightarrow \id_\mathcal{T}$. The cone of this morphism, denoted by $T_\mathbf{E}$, is called the \emph{twist functor} associated with $\mathbf{E}$. Our assumption that $\mathcal{T}$ is algebraic guarantees us that $T_\mathbf{E}$ is a well-defined endofunctor of $\mathcal{T}$; see \cite[Subsection 3.1]{hochenegger-kalck-ploog} for details.
\end{definition}

\begin{remark}
	We will use the following isomorphism for calculating $F_\mathbf{E}$ for a given object $X\in\mathcal{T}$:
	\[
		F_\mathbf{E}(X)\cong \bigoplus_{i=1}^n\bigoplus_{k\in\mathbb{Z}} \Sigma^kE_i^{\oplus \dim_\mathbb{K}\Hom_{\mathcal{T}}(\Sigma^kE_i,X)}.
	\]
	Note that our assumption on Hom-spaces guarantees that there are only finitely many non-zero direct summands in the above direct sum.
\end{remark}

We can easily calculate the action of a twist functor on the objects of its underlying exceptional cycle:
\begin{lemma}\label{te-on-exceptional-cycles}
	Assume $\mathbf{E}=(E_1,\ldots, E_n)$ is an exceptional $n$-cycle in $\mathcal{T}'$ with $k_\mathbf{E}=(k_1,\ldots, k_n)$. Then for $i=1,\ldots, n$ we have
	\[
		T_{\mathbf{E}}(E_i)=\Sigma^{-k_{i-1}+1}E_{i-1},
	\]
	where $E_0\coloneqq E_n$.
\end{lemma}
\begin{proof}
	This follows from arguments used in Step 1 of the proof of \cite[Theorem 4.5]{broomhead}.
\end{proof}

Let $\mathcal{S}$ be a subcategory of $\mathcal{T}$. Denote by $\mathcal{S}^\perp$ its \emph{right orthogonal}, that is, the full subcategory of $\mathcal{T}$ consisting of the objects $X\in\mathcal{T}$ with $\Hom_{\mathcal{T}}(\mathcal{S}, X)=0$. Dually, denote by $^\perp\mathcal{S}$ the \emph{left orthogonal} to $\mathcal{S}$, the full subcategory of $\mathcal{T}$ consisting of the objects $X\in\mathcal{T}$ with $\Hom_{\mathcal{T}}(X,\mathcal{S})=0$.

The twist functors have the following property, which follows from arguments used in Steps 1 and 2 of the proof of \cite[Theorem 4.5]{broomhead}:
\begin{proposition}\label{fully-faithful}
	Let $\mathbf{E}=(E_1,\ldots, E_n)$ be an exceptional $n$-cycle in $\mathcal{T}'$. Assume that $\Hom_{\mathcal{T}}(A, B)\cong D\Hom_{\mathcal{T}}(B, SA)$ for all $A\in\mathcal{T}'$, $B\in\mathcal{T}$, where $D$ denotes the standard duality. Then $T_\mathbf{E}$ is fully faithful on $\mathcal{T}$.\qed
\end{proposition}
This allows us to state the following:
\begin{corollary}\label{twist-functors-are-equivalences}
	Let $\mathbf{E}$ be as above. Assume that $\Hom_{\mathcal{T}}(A, B)\cong D\Hom_{\mathcal{T}}(B, SA)$ for all $A\in\mathcal{T}'$, $B\in\mathcal{T}$.
	\begin{enumerate}[noitemsep, label=(\arabic*)]
		\item If $\mathcal{T}'=\mathcal{T}$, then $T_\mathbf{E}$ is an autoequivalence.
		\item If $\mathcal{T}'^\perp = 0 = \, ^\perp\mathcal{T}'$ and $S\circ T_\mathbf{E}|_{\mathcal{T}'}\equiv T_\mathbf{E}|_{\mathcal{T}'}\circ S$, then $T_\mathbf{E}$ is an autoequivalence.
	\end{enumerate}
\end{corollary}
\begin{proof}
	First, define $\mathcal{E}\coloneqq \{\Sigma^l E_i\;|\; l\in\mathbb{Z}, i=1,\ldots, n\}$.

	(1) By Proposition~\ref{fully-faithful} it suffices to show that $T_\mathbf{E}$ is dense. Put $\Omega\coloneqq\mathcal{E}\cup\mathcal{E}^\perp$. Simple calculations allow  to see that $\Omega^\perp = 0$ and $^\perp\Omega=0$. By following the proof of Step 3 of \cite[Theorem 4.5]{broomhead} we can deduce that $T_\mathbf{E}$ is in fact an autoequivalence of $\mathcal{T}$.

	(2) Observe that by (1) the functor $T_\mathbf{E}|_{\mathcal{T}'}$ is an autoequivalence of $\mathcal{T}'$ and by Proposition~\ref{fully-faithful} is fully faithful on $\mathcal{T}$. Our assumptions allow to use \cite[Corollary 1.56]{huybrechts} and obtain the desired result.
\end{proof}

\subsection{Derived autoequivalences of discrete algebras}\label{section-25}

The derived discrete algebras have been introduced by Vossieck in \cite{vossieck}. A finite-dimensional $\mathbb{K}$-algebra $\Lambda$ is called \emph{derived discrete} if for each vector $\mathbf{h}\in(\mathbb{Z}_{\geq 0})^{\mathbb{Z}}$ there is only a finite number of isomorphism classes of objects of $\mathbf{D}^\mathrm{b}(\mod\Lambda)$ with cohomology dimension vector $\mathbf{h}$.

By \cite[Theorem A]{bobinski} any derived discrete algebra which is not piecewise hereditary of Dynkin type is derived equivalent to an algebra $\Lambda\coloneqq \Lambda(p,q,r)$ which is the path algebra of the gentle bound quiver depicted in Figure \ref{quiver-figure}. Moreover, $\gldim \Lambda =\infty$ if and only if $r=p$. The indecomposable simple and projective modules over $\Lambda(p,q,r)$ associated with the vertices $i=-q,\ldots, p-1$ will be denoted by $S_i$ and $P_i$, respectively.
\begin{figure}[H]
	\begin{center}
		\begin{tikzcd}
			&&&&1\arrow[ld,"\alpha_0"{name=zero}]&\arrow[l,"\alpha_1"]\cdots&\arrow[l,"\alpha_{p-r-2}"]p-r-1\\
			-q&\arrow[l,"\alpha_{-q}"]\cdots&\arrow[l,"\alpha_{-2}"]-1&\arrow[l,"\alpha_{-1}"]0\arrow[rd,"\alpha_{p-1}"{name=p1}]&&&&p-r\arrow[lu,"\alpha_{p-r-1}"]\\
			&&&&p-1\arrow[r,"\alpha_{p-2}"{name=p2}]&\cdots\arrow[r,"\alpha_{p-r+1}"{name=pr1}]&p-r+1\arrow[ru,"\alpha_{p-r}"{name=pr}]
			\arrow[r, dotted, no head, bend left=30, from=p1, to=p2]
			\arrow[r, dotted, no head, bend left=30, from=pr1, to=pr]
			\arrow[r, dotted, no head, bend left=30, shift right=3, from=zero, to=p1]
		\end{tikzcd}
	\end{center}
	\caption{\label{quiver-figure}The gentle bound quiver of $\Lambda(p,q,r)$. Here, the dotted arcs represent the zero-relations.}
\end{figure}
We have the following result:

\begin{proposition}\label{except-cycles-in-gldim-finite}
	Let $\Lambda=\Lambda(p,q,r)$, where $r<p$. Then
	\begin{equation*}
		\begin{split}
			\mathbf{X} &= (\Sigma^{r-1}P_{p-r+1}, \ldots, \Sigma P_{p-1}, P_0/P_{-1}, \ldots, P_{-q+1}/P_{-q}, P_{-q}),\\
			\mathbf{Y} &= (S_{p-r}, \ldots, S_{1}),
		\end{split}
	\end{equation*}
	are, respectively, $(q+r)$- and $(p-r)$-exceptional cycles in  $\mathbf{D}^\mathrm{b}(\mod\Lambda)$ with $k_\mathbf{X}=(1,\ldots, 1, -r+1)$ and $k_\mathbf{Y}=(1,\ldots,1,r+1)$.
\end{proposition}
\begin{remark}
	For brevity, we will denote by $X_i$, where $i=1,\ldots,q+r$, the $i$-th element of the exceptional cycle $\mathbf{X}$ and by $Y_j$, where $j=1,\ldots, p-r$, the $j$-th element of the exceptional cycle $\mathbf{Y}$. We will follow this convention throughout the paper.
\end{remark}
\begin{proof}

	We can easily see that for each object $X_i$ and $Y_j$ its graded endomorphism algebra in $\mathbf{D}^\mathrm{b}(\mod\Lambda)$ is one-dimensional, which shows the property (1) from Definition~\ref{definition-exceptional-cycle}. To show (2) we first need to study the action of the Serre functor on elements of the cycles $\mathbf{X}$ and $\mathbf{Y}$. By \cite[Proposition I.2.3]{reiten} we have $S(A)=\Sigma\tau(A)$ for $A\in\mathbf{D}^\mathrm{b}(\mod\Lambda)$, where $\tau$ is the Auslander-Reiten translation.

	Up to orientation of the arrows in the Gabriel quiver of $\Lambda$, the action of $\tau$ on the objects $X_i$ and $Y_j$ is given in \cite[Section 3]{bobinski}, namely
	\[
		\tau(X_i)=\Sigma^{m_i}X_{i+1} \hspace{1cm}\text{ and }\hspace{1cm} \tau(Y_j)=\Sigma^{n_j}Y_{j+1},
	\]
	where $X_{q+r+1}=X_1, Y_{p-r+1}=Y_1$ and
	\[
		m_i\coloneqq\begin{cases*}
			0  & if $i=1,\ldots, q+p-1$, \\
			-r & if $i=q+p$,
		\end{cases*},\hspace{0.5cm}
		n_j\coloneqq\begin{cases*}
			0 & if $j=1,\ldots, p-r-1$, \\
			r & if $j=p-r$.
		\end{cases*}
	\]
	This allows us to show (2) and that $k_\mathbf{X}=(1,\ldots, 1,-r+1)$ and $k_\mathbf{Y}=(1,\ldots, 1,r+1)$. We leave the proof of property (3) of Definition~\ref{definition-exceptional-cycle} to the reader.
\end{proof}

The results obtained in \cite{broomhead} allow us to state the following fact:
\begin{proposition}\label{automorphisms-gldim-finite}
	Let $\Lambda=\Lambda(p,q,r)$, where $r<p$. Then:
	\begin{enumerate}[noitemsep, label=(\arabic*)]
		\item For any derived autoequivalence $F$ of $\Lambda$ there exist $k,l,s\in\mathbb{Z}$ such that $F\equiv T_\mathbf{X}^kT_\mathbf{Y}^l\Sigma^s$.
		\item $T_\mathbf{X}$ and $T_\mathbf{Y}$ commute with each other.
		\item $\Sigma^r\equiv T_\mathbf{X}^{r+q}T_\mathbf{Y}^{r-p}$.
	\end{enumerate}
\end{proposition}
\begin{proof}
	A description of the group $\daut(\Lambda)$ of derived automorphisms of $\Lambda$ is given in \cite[Theorem 5.7]{broomhead}. Namely, if $F\in\daut(\Lambda)$, then there exist $k,l,s\in\mathbb{Z}$ such that $F\equiv T_\mathcal{X}^kT_\mathcal{Y}^l\Sigma^s$. Moreover, $\daut(\Lambda)$ is abelian and the relation $\Sigma^r\equiv T_\mathcal{X}^{r+q}T_\mathcal{Y}^{r-p}$ holds. The functors $T_\mathcal{X}$ and $T_\mathcal{Y}$ are defined in \cite[Section 5]{broomhead} to be the twist functors associated with, respectively, $(q+r)$- and $(p-r)$-exceptional cycles consisting of consecutive objects lying on the mouths of components $\mathcal{X}^0$ and $\mathcal{Y}^0$ of the Auslander-Reiten quiver of $\mathbf{D}^\mathrm{b}(\mod\Lambda)$.

	Note that \cite[Lemma 3.1 and Lemma 3.2]{bobinski} give us a description of the objects lying on the mouths of the components of the Auslander-Reiten quiver of $\mathbf{D}^\mathrm{b}(\mod\Lambda)$. We can deduce that $\mathbf{X}$ consists of $q+r$ consecutive objects lying on the mouth of one of the components $\mathcal{X}^0, \ldots, \mathcal{X}^{r-1}$ and $\mathbf{Y}$ consists of $p-r$ consecutive objects lying on the mouth of one of the components $\mathcal{Y}^0, \ldots, \mathcal{Y}^{r-1}$, where $\mathcal{X}^i\coloneqq\Sigma^i\mathcal{X}^0$ and $\mathcal{Y}^i\coloneqq\Sigma^i\mathcal{Y}^0$. By \cite[Lemma 5.3]{broomhead} (and the remark preceeding it) we infer that $T_\mathcal{X}=T_{\mathbf{X}}$ and $T_\mathcal{Y}=T_{\mathbf{Y}}$.
\end{proof}

The description of the derived autoequivalences in Proposition \ref{automorphisms-gldim-finite} can be generalized to arbitrary derived discrete algebra, which is not piecewise hereditary of Dynkin type and of finite global dimension. Let $\Lambda$ be such an algebra. There exists an equivalence $H\colon\mathbf{D}^\mathrm{b}(\mod\Lambda(p,q,r))\rightarrow\mathbf{D}^\mathrm{b}(\mod\Lambda)$ for $0\leq q$ and $1\leq r<p$. We can now put $T_{\mathbf{X}, \Lambda}\coloneqq HT_\mathbf{X}H^{-1}$ and $T_{\mathbf{Y}, \Lambda}\coloneqq HT_\mathbf{Y}H^{-1}$, where $H^{-1}$ is a quasi-inverse of $H$. By \cite[Lemma 4.1]{broomhead} we deduce that $T_{\mathbf{X}, \Lambda}$ and $T_{\mathbf{Y}, \Lambda}$ are well-defined and do not depend on the choice of an equivalence $H$ (and $H^{-1}$). In fact, $T_{\mathbf{X},\Lambda}=T_{H(\mathbf{X})}$, $T_{\mathbf{Y},\Lambda}=T_{H(\mathbf{Y})}$.
\begin{proposition}\label{automorphisms-gldim-finite-arbitrary}
	Let $\Lambda$ be a derived discrete algebra, not piecewise hereditary of Dynkin type and of finite global dimension and assume that $\Lambda$ is derived equivalent to $\Lambda(p,q,r)$ for $0\leq q$, $1\leq r<p$. Then for any derived autoequivalence $F$ of $\Lambda$ there exist $k,l,s\in\mathbb{Z}$ such that $F\equiv T_{\mathbf{X},\Lambda}^kT_{\mathbf{Y},\Lambda}^l\Sigma^s$. 
\end{proposition}
\begin{proof}
	Assume that $F$ is a derived autoequivalence of $\Lambda$ and pick an equivalence\\ $H\colon\mathbf{D}^\mathrm{b}(\mod\Lambda(p,q,r))\rightarrow\mathbf{D}^\mathrm{b}(\mod\Lambda)$. This means that $H^{-1}FH$ is an autoequivalence of $\mathbf{D}^\mathrm{b}(\mod\Lambda(p,q,r))$ and by Proposition \ref{automorphisms-gldim-finite}(1) there exist $k,l,s\in\mathbb{Z}$ such that $H^{-1}FH\equiv T_{\mathbf{X}}^kT_\mathbf{Y}^l\Sigma^s$. Consequently,
	\begin{equation*}
		\begin{split}
			F\equiv HH^{-1}FHH^{-1}\equiv HT_{\mathbf{X}}^kT_\mathbf{Y}^l\Sigma^sH^{-1}\equiv (HT_{\mathbf{X}}H^{-1})^k(HT_\mathbf{Y}H^{-1})^l(H\Sigma H^{-1})^s\equiv T_{{\mathbf{X}},\Lambda}^kT_{\mathbf{Y},\Lambda}^l\Sigma^s.\qedhere
		\end{split}
	\end{equation*}
\end{proof}

In the case of algebras $\Lambda(p,q,p)$, i.e.\ those of infinite global dimension we now have the following:
\begin{proposition}
	Let $\Lambda=\Lambda(p,q,p)$. Then
	\[
		\mathbf{X}\coloneqq (\Sigma^{p-1}P_{1}, \ldots, \Sigma P_{p-1}, P_0/P_{-1},\ldots, P_{-q+1}/P_{-q}, P_{-q})
	\]
	is an exceptional $(p+q)$-cycle in $\per(\Lambda)$ with $k_\mathbf{X}=(1,\ldots, 1, -p+1)$. Moreover, for $\mathcal{T}\in\{\per(\Lambda), \mathbf{D}^\mathrm{b}(\mod\Lambda)\}$, the twist functor $T_\mathbf{X}\colon\mathcal{T}\rightarrow\mathcal{T}$ is an autoequivalence of $\mathcal{T}$.
\end{proposition}
\begin{proof}
	Since algebras of the form $\Lambda(p,q,r)$ are gentle and therefore Gorenstein (\cite[Theorem 3.4]{geiss-reiten}), the category $\per(\Lambda)$ admits a Serre functor $S$ with $S(A)=\Sigma\tau(A)$ for all $A\in\per(\Lambda)$ by \cite[Proposition I.2.3]{reiten}. The fact that $\mathbf{X}$ is an exceptional $(p+q)$-cycle in $\per(\Lambda)$ can be shown in an analogous way as in the finite global dimension case in Proposition~\ref{except-cycles-in-gldim-finite}.

	By \cite[Lemma 6.4.5 and Theorem 6.4.6]{krause} we know that for all $A\in \per(\Lambda)$ and $B\in\mathbf{D}^\mathrm{b}(\mod\Lambda)$ we have
	\[
		\Hom_{\mathbf{D}^\mathrm{b}(\mod\Lambda)}(A,B)\cong D\Hom_{\mathbf{D}^\mathrm{b}(\mod\Lambda)}(B,SA).
	\]
	We can now apply Corollary~\ref{twist-functors-are-equivalences}(1) for $\mathcal{T} = \mathcal{T}'\coloneqq \per(\Lambda)$ and deduce that $T_\mathbf{X}$ is an autoequivalence of $\per(\Lambda)$. Similarly to \cite[Corollary 5.5]{broomhead} we can show that that $T_\mathbf{X}\equiv\tau^{-1}$, which implies $S\circ T_\mathbf{X}\equiv T_\mathbf{X}\circ S$.

	Put now $\mathcal{T}'\coloneqq\per(\Lambda)$, $\mathcal{T}\coloneqq\mathbf{D}^\mathrm{b}(\mod\Lambda)$, and consider $T_\mathbf{X}$ as a functor on $\mathcal{T}$. We can see that if for $B\in\mathbf{D}^\mathrm{b}(\mod\Lambda)$ we have $0=\Hom_{\mathbf{D}^\mathrm{b}(\mod\Lambda)}(\Lambda, \Sigma^iB)\cong H^i(B)$ for all $i\in\mathbb{Z}$, then $B$ is quasi-isomorphic to the zero complex. But this is just equivalent to stating $\mathcal{T}'^\perp = 0$. Assume now that $B\in\,^\perp\mathcal{T}'$. This means in particular that for all $i\in\mathbb{Z}$ we have
	\[
		0=\Hom_{\mathbf{D}^\mathrm{b}(\mod\Lambda)}(\Sigma^i B, S\Lambda)\cong D\Hom_{\mathbf{D}^\mathrm{b}(\mod\Lambda)}(\Lambda, \Sigma^i B)\cong DH^i(B),
	\]
	hence again $B\cong 0$, thus $^\perp\mathcal{T}'=0$. We can apply Corollary~\ref{twist-functors-are-equivalences}(2) and deduce that $T_{\mathbf{X}}$ is an autoequivalence of $\mathbf{D}^\mathrm{b}(\mod\Lambda)$.
\end{proof}

Using the above observation, we can formulate the following statement regarding auto\-equivalences of $\per(\Lambda(p,q,p))$ and $\mathbf{D}^\mathrm{b}(\mod\Lambda)$:
\begin{proposition}\label{automorphisms-gldim-infinite}
	Let $\Lambda=\Lambda(p,q,p)$. Then for $\mathcal{T}\in \{\per(\Lambda), \mathbf{D}^\mathrm{b}(\mod\Lambda)\}$ the following hold:
	\begin{enumerate}[noitemsep, label=(\arabic*)]
		\item For any autoequivalence $F$ of $\mathcal{T}$ there exist $k,s\in\mathbb{Z}$ such that $F\equiv T_\mathbf{X}^k\Sigma^s$.
		\item $\Sigma^p\equiv T_\mathbf{X}^{p+q}$.
	\end{enumerate}
\end{proposition}
\begin{proof}
	The claim is shown in an analogous way as steps (2)-(4) of the proof of \cite[Theorem 5.7]{broomhead}.
\end{proof}

Similarly to Proposition \ref{automorphisms-gldim-finite-arbitrary}, we can also generalize Proposition \ref{automorphisms-gldim-infinite} the same way. If $\Lambda$ is a derived discrete algebra of infinite global dimension, then there exists an equivalence $H\colon \mathbf{D}^\mathrm{b}(\mod\Lambda(p,q,p))\rightarrow\mathbf{D}^\mathrm{b}(\mod\Lambda)$ for $0\leq q$, $1\leq p$. Using the same arguments as in the paragraph above Proposition \ref{automorphisms-gldim-finite} we introduce a well-defined derived autoequivalence $T_{\mathbf{X},\Lambda}\coloneqq HT_\mathbf{X}H^{-1}$ of $\Lambda$. Note that by \cite[Proposition 6.2]{rickard} any derived equivalence $\mathbf{D}^\mathrm{b}(\mod\Lambda_1)\cong\mathbf{D}^\mathrm{b}(\mod\Lambda_2)$ necessarily restricts to an equivalence $\per(\Lambda_1)\cong\per(\Lambda_2)$, for any $\mathbb{K}$-algebras $\Lambda_1,\Lambda_2$ and thus we can define the equivalence $T_{\mathbf{X},\Lambda}\colon\per(\Lambda)\rightarrow\per(\Lambda)$ analogously.
\begin{proposition}\label{automorphisms-gldim-infinite-arbitrary}
	Let $\Lambda$ be an arbitrary derived discrete algebra of infinite global dimension and assume that $\Lambda$ is derived equivalent to $\Lambda(p,q,p)$ for some $0\leq q$ and $1\leq p$. Suppose that $\mathcal{T}\in \{\per(\Lambda), \mathbf{D}^\mathrm{b}(\mod\Lambda)\}$. Then for any autoequivalence $F$ of $\mathcal{T}$ there exist $k,s\in\mathbb{Z}$ such that $F\equiv T_{\mathbf{X}, \Lambda}^k\Sigma^s$.\qed
\end{proposition}

\section{Finite global dimension}\label{section-3}

The following general result will be useful in our considerations:
\begin{lemma}\label{limit-lemma}
	Let $(A_n)_{n>0}$ be a sequence of non-empty finite multisubsets of real numbers of cardinality growing at most polynomially with respect to $n$. Then
	\[
		\lim_{n\to\infty}\frac{1}{n}\log\left(\sum_{a\in A_n} \exp(a)\right)=\lim_{n\to\infty}\frac{1}{n}\max_{a\in A_n}(a).
	\]
\end{lemma}
\begin{proof}
	\pushQED{\qed}
	We have the following inequality
	\begin{equation*}
		\begin{split}
			\frac{1}{n}\log\left(\sum_{a\in A_n}\exp(a)\right)&\leq \frac{1}{n}\log(|A_n|\cdot\max_{a\in A_n}(\exp(a)))=\frac{\log(|A_n|)+\max_{a\in A_n}(a)}{n}.
		\end{split}
	\end{equation*}
	On the other hand, we have a lower bound
	\[
		\frac{1}{n}\log\left(\sum_{a\in A_n}\exp(a)\right)\geq \frac{1}{n}\log(\exp(\max_{a\in A_n}(a)) = \frac{1}{n}\max_{a\in A_n}(a).
	\]
	Now the claim follows from the squeeze theorem since $\lim_{n\to\infty}\frac{\log(|A_n|)}{n}=0$.
\end{proof}
We will also need a result regarding non-zero cohomologies in indecomposable complexes in the derived category:
\begin{lemma}
	Let $\Lambda$ be a finite-dimensional $\mathbb{K}$-algebra of finite global dimension. Let $X\in\mathbf{D}^\mathrm{b}(\mod\Lambda)$ be an object such that $H^a(X)\neq 0 \neq H^b(X)$ for some $a<b$, where $b-a>\gldim\Lambda$ and for all $i$ such that $a<i<b$ we have $H^i(X)=0$. Then $X$ is not indecomposable.
\end{lemma}
\begin{proof}
	Without loss of generality we can assume $X$ is of the form
	\[
		0\rightarrow X^k\xrightarrow{d^k}\ldots \xrightarrow{} X^a\xrightarrow{d^a} X^{a+1}\xrightarrow{d^{a+1}}\ldots\rightarrow X^{b-1}\xrightarrow{d^{b-1}} X^b\xrightarrow{d^b} \ldots \xrightarrow{} X^l\rightarrow 0
	\]
	for some $k<l$, where $X^i\in\proj\Lambda$ for all $i=k,\ldots, l$. Put $m\coloneqq \pd(X^b/\im(d^{b-1}))\leq\gldim\Lambda$. The object $X^b/\im(d^{b-1})$ is quasi-isomorphic to its projective resolution via the following morphism of complexes
	\begin{center}
		\begin{tikzcd}
			0\arrow[r]&P^{b-m}\arrow[r,"\partial^{b-m}"]&\ldots\arrow[r]&P^{b-1}\arrow[r,"\partial^{b-1}"]&P^b\arrow[d,"p"]\arrow[r]&0\\
			&&&&X^b/\im(d^{b-1})
		\end{tikzcd},
	\end{center}
	where $p$ is an epimorphism. By the proof of the Comparison Theorem we have the following morphisms between truncated complexes
	\begin{center}
		\begin{tikzcd}
			\ldots\arrow[r]&X^{b-m-1}\arrow[d, shift left, "0"]\arrow[r]&X^{b-m}\arrow[d,shift left, "f_{b-m}"]\arrow[r,"d^{b-m}"]&\ldots\arrow[r]&X^{b-1}\arrow[r,"d^{b-1}"]\arrow[d,shift left,"f_{b-1}"]&X^b\arrow[d,shift left,"f_{b}"]&\\
			\ldots\arrow[r]&0\arrow[u,shift left, "0"]\arrow[r]&P^{b-m}\arrow[r,"\partial^{b-m}"]\arrow[u,shift left,"g_{b-m}"]&\ldots\arrow[r]&P^{b-1}\arrow[r,"\partial^{b-1}"]\arrow[u,shift left,"g_{b-1}"]&P^b\arrow[u,shift left, "g_b"]
		\end{tikzcd}
	\end{center}
	together with maps $s_i\colon P_i\rightarrow P_{i-1}$ for $i=b-m+1, \ldots, b$ such that $f_i\circ g_i-\id_{P^i}=s_{i+1}\circ\partial^i+\partial^{i-1}\circ s_i$,\; $p\circ f_b=\pi$ and $\pi\circ g_b=p$, where we put $s_{b+1}\coloneqq 0$, and $\pi\colon X^b\rightarrow X^b/\im(d^{b-1})$ is the canonical projection. Here we use the fact that $H^i(X)=0$ for $(a<)\,b-m\leq i<b$.

	Let $\overline{d^b}\colon X^b/\im(d^{b-1})\rightarrow X^{b+1}$ be the map induced by $d^b$, i.e.\ $\overline{d^b}\circ\pi=d^b$. If $\alpha\coloneqq \overline{d^b}\circ p$ then $\alpha\circ f_b=d^b$ and $d^b\circ g_b=\alpha$. Therefore, we have the following morphisms of complexes:
	\begin{center}
		\begin{tikzcd}
			X\colon \arrow[d,shift left,"f"]&\ldots\arrow[r]&X^{b-m-1}\arrow[d, shift left, "0"]\arrow[r]&X^{b-m}\arrow[d,shift left, "f_{b-m}"]\arrow[r,"d^{b-m}"]&\ldots\arrow[r]&X^{b-1}\arrow[r,"d^{b-1}"]\arrow[d,shift left,"f_{b-1}"]&X^b\arrow[r,"\pi"]\arrow[d,shift left,"f_{b}"]&X^{b+1}\arrow[d,equal]\arrow[r]&\ldots\\
			Y\colon\arrow[u,shift left,"g"]&\ldots\arrow[r]&0\arrow[u, shift left, "0"]\arrow[r]&P^{b-m}\arrow[r,"\partial^{b-m}"]\arrow[u,shift left,"g_{b-m}"]&\ldots\arrow[r]&P^{b-1}\arrow[r,"\partial^{b-1}"]\arrow[u,shift left,"g_{b-1}"]&P^b\arrow[r,"\alpha"]\arrow[u,shift left,"g_{b}"]&X^{b+1}\arrow[r]&\ldots
		\end{tikzcd}
	\end{center}
	Using the maps $s_i$ for $i=b-m+1,\ldots, b+1$ defined earlier, we see that $f\circ g\sim \id_Y$. This means that $Y$ is a direct summand of $X$. However, we know that $Y\neq 0$, since $H^b(Y)\neq 0$, and $Y\not\cong X$ as $H^a(Y)=0$ and $H^a(X)\neq 0$. Thus $X$ is not indecomposable.
\end{proof}
We immediately obtain the following fact:
\begin{proposition}\label{number-of-nontrivial-homologies}
	Let $\Lambda$ be as above and $X$ an indecomposable object in $\mathbf{D}^{[a,b]}(\mod\Lambda)$ such that $H^a(X)\neq 0 \neq H^b(X)$. Then the number of non-zero cohomologies in $X$ is at least $\frac{b-a}{\gldim\Lambda}$.\qed
\end{proposition}

Throughout the rest of this section we assume that $\Lambda=\Lambda(p,q,r)$, where $r<p$. Note that the finite global dimension assumption guarantees that $\Lambda$ is a split-generator of $\mathbf{D}^\mathrm{b}(\mod\Lambda)$. We will use it to calculate both the entropy and the polynomial entropy of the derived autoequivalences of $\Lambda$.

\subsection{Twist functors}

We will start this section with calculating projective resolutions of objects in the exceptional cycles $\mathbf{X}$ and $\mathbf{Y}$. Observe first that for $i=-q+1,\ldots, 0$ (assuming that $q>0$), by taking the projective resolutions of the objects $P_{i}/P_{i-1}$ we have the following obvious quasi-isomorphisms
\begin{center}
	\begin{tikzcd}[column sep=small, row sep=small]
		P_{i-1}\arrow[r,"\alpha_{i-1}^*"]&P_i\arrow[d, "\pi"]\\
		& P_i/P_{i-1}
	\end{tikzcd}.
\end{center}
Here $\alpha_{i}^*$'s are the morphisms induced by the paths $\alpha_{-i}$ in the quiver of the algebra $\Lambda$ and $\pi$'s are the canonical projections on respective cokernels.

Moreover, observe that the modules $S_i$ for $i=1,\ldots, p-r$ are not projective and we have the following quasi-isomorphisms:
\begin{center}
	\begin{tikzcd}
		P_{i-1}\arrow[r,"\alpha_{i-1}^*"]&P_{i}\arrow[d, "\pi"]\\
		& S_i
	\end{tikzcd}\hspace{0.5cm} for $i=2,\ldots, p-r$, and
	\begin{tikzcd}
		P_{p-r}\arrow[r,"\alpha_{p-r}^*"]&P_{p-r+1}\arrow[r, "\alpha_{p-r+1}^*"]&\ldots\arrow[r]&P_0\arrow[r,"\alpha_{0}^*"]&P_1\arrow[d, "\pi"]\\
		&&&&S_1
	\end{tikzcd}.
\end{center}

We can also calculate the action of $T_\mathbf{X}$ and $T_\mathbf{Y}$ on the objects of their respective underlying exceptional cycles:
\begin{lemma}\label{tx-ty-action-on-xy}
	\begin{enumerate}[noitemsep, label=(\arabic*)]
		\item If $n=(q+r)k+l$ for $k\in\mathbb{Z}$ and $l\in\{0,\ldots, q+r-1\}$, then
		      \[
			      T_\mathbf{X}^n(X_i)=\begin{cases*}
				      \Sigma^{kr}X_{i-l}           & if $i-l\geq 1$, \\
				      \Sigma^{kr+1}X_{(q+r)-(i-l)} & otherwise.
			      \end{cases*}
		      \]
		\item If $n=(p-r)k+l$ for $k\in\mathbb{Z}$ and $l\in\{0,\ldots, p-r-1\}$, then
		      \[
			      T_\mathbf{Y}^n(Y_i)=\begin{cases*}
				      \Sigma^{-kr}Y_{i-l}           & if $i-l\geq 1$, \\
				      \Sigma^{-kr-1}Y_{(p-r)-(i-l)} & otherwise.
			      \end{cases*}
		      \]
	\end{enumerate}
\end{lemma}
\begin{proof}
	This is a direct consequence of Lemma~\ref{te-on-exceptional-cycles} and Proposition~\ref{except-cycles-in-gldim-finite}.
\end{proof}

\begin{lemma}\label{min-max-homologies-lemma}
	Let $n\geq 0$. Then:
	\begin{enumerate}[noitemsep, label=(\arabic*)]
		\item $T_\mathbf{Y}^n(\Lambda)\subseteq \mathbf{D}^{[0, r\cdot\ceil{\frac{n}{p-r}}]}(\mod\Lambda)$.
		\item $\min(i\in\mathbb{Z}\;|\; H^i(T_\mathbf{Y}^n(P_{p-r})))=0$ and $\max(i\in\mathbb{Z}\;|\; H^i(T_\mathbf{Y}^n(P_{p-r})))=r\cdot\ceil{\frac{n}{p-r}}$.
		\item $T_\mathbf{X}^n(\Lambda)\subseteq\mathbf{D}^{[-r\cdot\ceil{\frac{n}{q+r}},0]}(\mod\Lambda)$.
		\item $\min(i\in\mathbb{Z}\;|\; H^i(T_\mathbf{X}^n(P_{1})))=-r\cdot\floor{\frac{n}{q+r}}-l$ for some $l\in\{0,\ldots, r-1\}$ \\and $\max(i\in\mathbb{Z}\;|\; H^i(T_\mathbf{X}^n(P_{1})))=0$.
	\end{enumerate}
\end{lemma}
\begin{proof}

	(1), (2) We begin with objects $P_i$ with $i\notin\{1,\ldots, p-r\}$. Using the projective resolutions of the objects $Y_j$ for $j=1,\ldots, p-r$ we can calculate that $\Hom_{\mathbf{D}^\mathrm{b}(\mod\Lambda)}(\Sigma^lY_j, P_i)=0$ for all $l\in\mathbb{Z}$, therefore  $F_\mathbf{Y}(P_i)=0$, which implies that $T_\mathbf{Y}(P_i)=P_i$. This allows us to deduce that $T_\mathbf{Y}^n(P_i)\in\mathbf{D}^{[0,0]}(\mod\Lambda)$ for any $n\geq 0$.

	Moreover, we can see that $F_\mathbf{Y}(P_i)=\Sigma^{-1}S_{i+1}$ and $T_\mathbf{Y}(P_i)=P_{i+1}$ for $i=1,\ldots, p-r-1$. Indeed, the only morphism from the objects $\Sigma^lY_j$ to $P_i$ is the one of the form
	\begin{center}
		\begin{tikzcd}
			P_i\arrow[r, "\alpha_i^*"]\arrow[d,equal]& P_{i+1}\\
			P_i
		\end{tikzcd},
	\end{center}
	whose cone is homotopy equivalent to $P_{i+1}$.

	Therefore, it suffices to look at the cohomologies of the objects $T_\mathbf{Y}^n(P_{p-r})$. Observe that $F_\mathbf{Y}(P_{p-r})=\Sigma^{-r-1}S_1$ using the same argument as above. Hence, from the exact triangle
	\begin{equation}\label{pqp-triangle-equation}
		F_\mathbf{Y}(P_{p-r})\rightarrow P_{p-r}\rightarrow T_\mathbf{Y}(P_{p-r})\rightarrow \Sigma F_\mathbf{Y}(P_{p-r})
	\end{equation}
	we infer by Lemma~\ref{triangle-tower-homologies}(3) that $T_\mathbf{Y}(P_{p-r})\in\mathbf{D}^{[0,r]}(\mod\Lambda)$ and $H^{j}(T_\mathbf{Y}(P_{p-r}))\neq 0$ for $j=0,r$. Moreover, for $l=1,\ldots, p-r$, by Lemma~\ref{tx-ty-action-on-xy}(2) we have $T_\mathbf{Y}^{l-1}(\Sigma^{-r-1}S_1)=\Sigma^{-r-1}S_{l}$. By applying $T_\mathbf{Y}^{l-1}$ to triangle~(\ref{pqp-triangle-equation}) and using Lemma~\ref{triangle-tower-homologies}(3) we deduce inductively that $T_\mathbf{Y}^l(P_{p-r})\in\mathbf{D}^{[0,r]}(\mod\Lambda)$ and $H^j(T_\mathbf{Y}^{l}(P_{p-r}))\neq 0$ for $j=0,r$. Next, $T_\mathbf{Y}^{p-r}(\Sigma^{-r-1}S_1)=\Sigma^{-2r-1}S_1$ and thus, by applying $T_\mathbf{Y}^{p-r}$ to triangle~(\ref{pqp-triangle-equation}) and using Lemma~\ref{triangle-tower-homologies}(3) we deduce that $T_\mathbf{Y}^{p-r+1}(P_{p-r})\in\mathbf{D}^{[0,2r]}(\mod\Lambda)$ and $H^{j}(T_\mathbf{Y}^{p-r+1}(P_{p-r}))\neq 0$ for $j=0,2r$.

	We now proceed similarly by induction using Lemma~\ref{triangle-tower-homologies}(3) and infer that for any $k\geq 0$ and $l\in\{1,\ldots, p-r\}$ we have $T_\mathbf{Y}^{k(p-r)+l}(P_{p-r})\in\mathbf{D}^{[0,(k+1)r]}$ and $H^{j}(T_\mathbf{Y}^{k(p-r)+l}(P_{p-r}))\neq 0$ for $j=0,(k+1)r$. This finishes the proof of both (1) and (2).

	(3) Due to Lemma~\ref{tx-ty-action-on-xy}(1) we know that for given $n\geq 0$ and $i\in\{-q\}\cup\{p-r+1,\ldots, p-1\}$ the object $T_\mathbf{X}^n(P_i)$ is a stalk complex with its only non-zero cohomology concentrated in degree $-r\cdot\floor{\frac{n}{q+r}}-l$ for some $l\in\{0,\ldots, r-1\}$ (note that if $r=1$, then $\{p-r+1,\ldots, p-1\}=\emptyset$). Thus, we need to study cohomologies for the objects $T_\mathbf{X}^n(P_i)$ for $i\in\{-q+1,\ldots, -1,0\}\cup\{1\ldots, p-r\}$ (keep in mind that if $q=0$, then $\{-q+1,\ldots, 0\}=\emptyset$).

	First, consider the objects $P_i$, where $i=1,\ldots, p-r$. We have an exact triangle
	\begin{equation}\label{pi-triangle-equation}
		F_\mathbf{X}(P_i)\rightarrow P_i\rightarrow T_\mathbf{X}(P_i)\rightarrow \Sigma F_\mathbf{X}(P_i).
	\end{equation}
	A direct calculation shows that $F_\mathbf{X}(P_i)=P_{-q}$. This follows from the fact that the only non-zero morphism from the objects of the form $\Sigma^lX_j$ for $j=1,\ldots,q+r$ and $l\in\mathbb{Z}$ is of the form $P_{-q}\xrightarrow{\alpha_{i-1}^*\ldots\alpha_{-q}^*}P_i$. Note also that for $k\geq 0$ we have $T_\mathbf{X}^{k(q+r)+l}(P_{-q})=\Sigma^{kr}(P_{-q+l}/P_{-q+l-1})$ for $l=1,\ldots, q$, and $T_\mathbf{X}^{k(q+r)+q+l}(P_{-q})=\Sigma^{kr+l}P_{p-l}$ for $l=1,\ldots, r-1$.

	Assume by induction that for some $k\geq 0$ the object $T_\mathbf{X}^{k(q+r)}(P_i)\in\mathbf{D}^{[-kr,0]}(\mod\Lambda)$. This is trivially true for $k=0$.
	By applying $T_\mathbf{X}^{k(q+r)+l}$ to triangle (\ref{pi-triangle-equation}) we deduce by induction for $l\in\{0,\ldots, q-1\}$ that $T_\mathbf{X}^{k(q+r)+l+1}(P_i)\in\mathbf{D}^{[-kr,0]}(\mod\Lambda)$
	by Lemma~\ref{triangle-tower-homologies}(2). Here we use the fact that $T_\mathbf{X}^{k(q+r)+l}(F_\mathbf{X}(P_i))$ is a shifted simple module concentrated in degree $-kr$ and that $T_\mathbf{X}^{k(q+r)+l}(P_i)$ and $T_\mathbf{X}^{k(q+r)+l+1}(P_i)$ are indecomposable. If $l\in\{q,\ldots, q+r-1\}$ then using Lemma~\ref{triangle-tower-homologies}(1) and applying $T_\mathbf{X}^{k(q+r)+l}$ to the triangle (\ref{tx-ty-action-on-xy}) we get by induction that $T_\mathbf{X}^{k(q+r)+l+1}(P_i)\in\mathbf{D}^{[-kr-(l+1-q),0]}(\mod\Lambda)$. In this way we show for $n\geq 0$ that $T_\mathbf{X}^n(P_i)\in\mathbf{D}^{[-r\cdot\ceil{\frac{n}{q+r}},0]}(\mod\Lambda)$.

	We will now consider the objects $P_i$, where $i=-q+1,\ldots, 0$ (assuming that $q>0$). Observe that we have a tower of exact triangles
	\begin{center}
		\begin{tikzcd}[row sep=small, column sep=small]
			P_{-q}\arrow[rr,"\alpha_{-q}^*"]&&P_{-q+1}\arrow[ld]\arrow[r]&\cdots\arrow[r]&P_{-1}\arrow[ld]\arrow[rr,"\alpha_{-1}^*"]&&P_0\arrow[ld]\\
			&P_{-q+1}/P_{-q}\arrow[ul,dashed]&&\arrow[ul,dashed]\cdots&&P_{0}/P_{-1}\arrow[ul,dashed]
		\end{tikzcd}
	\end{center}
	where $P_{-q}, P_{-q+1}/P_{-q},\ldots P_0/P_{-q}\in\mathbf{D}^{[0,0]}(\mod\Lambda)$. By applying $T_\mathbf{X}^n$ to it, we get a tower
	\begin{center}
		\begin{tikzcd}[row sep=small, column sep=small]
			T_\mathbf{X}^n(P_{-q})\arrow[rr,"T_\mathbf{X}^n(\alpha_{-q}^*)"]&&T_\mathbf{X}^n(P_{-q+1})\arrow[ld]\arrow[r]&\cdots\arrow[r]&T_\mathbf{X}^n(P_{-1})\arrow[ld]\arrow[rr,"T_\mathbf{X}^n(\alpha_{-1}^*)"]&&T_\mathbf{X}^n(P_0)\arrow[ld]\\
			&T_\mathbf{X}^n(P_{-q+1}/P_{-q})\arrow[ul,dashed]&&\arrow[ul,dashed]\cdots&&T_\mathbf{X}^n(P_{0}/P_{-1})\arrow[ul,dashed]
		\end{tikzcd}
	\end{center}
	We know by Lemma~\ref{tx-ty-action-on-xy}(1) that $T_\mathbf{X}^n(P_{-q}), T_\mathbf{X}^n(P_i/P_{-i-1})\in\mathbf{D}^{[-r\cdot\ceil{\frac{n}{q+r}}, 0]}(\mod\Lambda)$ for $i=-q+1,\\\ldots, 0$. By Lemma~\ref{triangle-tower-homologies}(1) we infer that $T_\mathbf{X}^n(P_i)\in\mathbf{D}^{[-r\cdot\ceil{\frac{n}{q+r}}, 0]}(\mod\Lambda)$ for $i=-q+1,\ldots, 0$. This allows us to finish the proof of (3).

	(4) Obviously, $\Hom_{\mathbf{D}^\mathrm{b}(\mod\Lambda)}(P_1, P_1)\neq 0$. Assume inductively that $\Hom_{\mathbf{D}^\mathrm{b}(\mod\Lambda)}(P_1, T_\mathbf{X}^n(P_1)) \neq 0$ for some $n\geq 0$. We have the following exact triangle:
	\[
		T_\mathbf{X}^nF_\mathbf{X}(P_1)\rightarrow T_\mathbf{X}^n(P_1)\rightarrow T_\mathbf{X}^{n+1}(P_1)\rightarrow\Sigma T_\mathbf{X}^nF_\mathbf{X}(P_1),
	\]
	to which we apply $\Hom_{\mathbf{D}^\mathrm{b}(\mod\Lambda)}(P_1, -)$. Observe that by our previous calculations $T_\mathbf{X}^nF_\mathbf{X}(P_1)$ is either a stalk complex concentrated in negative degree or one of the modules $P_{-q}$ or $P_{-j}/P_{-j-1}$ for $j=0,\ldots, -q+1$. Thus $\Hom_{\mathbf{D}^\mathrm{b}(\mod\Lambda)}(P_1, T_\mathbf{X}^nF_\mathbf{X}(P_1)))=0$. This and our induction hypothesis imply that $\Hom_{\mathbf{D}^\mathrm{b}(\mod\Lambda)}(P_1, T_\mathbf{X}^{n+1}(P_1))\neq 0$. Consequently
	\[
		H^0(T_\mathbf{X}^n(P_1))\cong \Hom_{\mathbf{D}^\mathrm{b}(\mod\Lambda)}(\Lambda, T_\mathbf{X}^n(P_1)) \neq 0
	\]
	for all $n\geq 0$.

	Observe now that $\Hom_{\mathbf{D}^\mathrm{b}(\mod\Lambda)}(P_{-q}, P_1)\neq 0$. Since $T_\mathbf{X}$ is an autoequivalence we infer that for any $n\geq 0$ we have $\Hom_{\mathbf{D}^\mathrm{b}(\mod\Lambda)}(T_\mathbf{X}^n(P_{-q}), T_\mathbf{X}^n(P_1))\neq 0$. Recall that by previous calculations we know that $T_\mathbf{X}^n(P_{-q})\in\mathbf{D}^{[-a_n,-a_n]}(\mod\Lambda)$ and $T_\mathbf{X}^n(P_1)\in\mathbf{D}^{[-a_n,0]}(\mod\Lambda)$, where $a_n= r\cdot\floor{\frac{n}{q+r}}+l_n$ for some $l_n\in\{0,\ldots, r-1\}$. If we had $H^{-a_n}(T_\mathbf{X}^n(P_1))=0$, then $T_\mathbf{X}^n(P_1)\in\mathbf{D}^{[-a_n+1,0]}(\mod\Lambda)$ and consequently $\Hom_{\mathbf{D}^\mathrm{b}(\mod\Lambda)}(T_\mathbf{X}^n(P_{-q}), T_\mathbf{X}^n(P_1))=0$, a contradiction. Thus $H^{-a_n}(T_\mathbf{X}^n(P_1))\neq 0$, which finishes the proof.
\end{proof}
\begin{remark}
	For the future calculations, note that we can rewrite the above expressions in the following way:
	\[
		\min(i\in\mathbb{Z}\;|\; H^i(T_\mathbf{X}^n(\Lambda)))=-r\cdot\ceil{\frac{n}{q+r}}+l=-r\left(\frac{n}{q+r}+\mu_1\right)+l=-\frac{rn}{q+r}-\mu_1r+l,
	\]
	\[
		\max(i\in\mathbb{Z}\;|\; H^i(T_\mathbf{Y}^n(\Lambda)))=r\cdot\ceil{\frac{n}{p-r}}=r\left(\frac{n}{p-r}+\mu_2\right)=\frac{rn}{p-r}+\mu_2r,
	\]
	where $l\in\{0,\ldots, r\}, \mu_1,\mu_2\in[0,1)$.
\end{remark}
We obtain the following fact:
\begin{corollary}\label{exactly-linear-number-homologies}
	The numbers of non-trivial cohomologies in the objects $T_\mathbf{X}^n(\Lambda)$ and $T_\mathbf{Y}^n(\Lambda)$ grow linearly with respect to $n$.
\end{corollary}
\begin{proof}
	By Lemma~\ref{min-max-homologies-lemma}(1),(3) we know that numbers of non-zero cohomologies in the objects $T_\mathbf{X}^n(\Lambda)$ and $T_\mathbf{Y}^n(\Lambda)$ grow at most linearly with respect to $n$. Moreover, both $T_\mathbf{X}^n(\Lambda)$ and $T_\mathbf{Y}^n(\Lambda)$ have indecomposable direct summands with number of cohomologies growing at least linearly with respect to $n$, which follows from Lemma~\ref{min-max-homologies-lemma}(2),(4) and Proposition~\ref{number-of-nontrivial-homologies}.
\end{proof}
\begin{proposition}\label{single-power-entropy}
	We have the following equalities:
	\[
		h_{T_\mathbf{X}}(t)=
		\begin{cases*}
			0              & if $t<0$,     \\
			\frac{r}{r+q}t & if $t\geq 0$,
		\end{cases*} \hspace{.7cm}and\hspace{.7cm}
		h_{T_\mathbf{Y}}(t)=
		\begin{cases*}
			\frac{r}{r-p}t & if $t<0$,     \\
			0              & if $t\geq 0$.
		\end{cases*}
	\]
\end{proposition}
\begin{proof}
	We have the following equality
	\[
		\delta_{\Lambda,T_\mathbf{E}^n(\Lambda)}'(t) = \sum_{a\in\mathbb{Z}}\dim_\mathbb{K}\Hom_{\mathbf{D}^\mathrm{b}(\mod\Lambda)}(\Lambda, \Sigma^aT_\mathbf{E}^n(\Lambda))\exp(-at)= \sum_{a\in\mathbb{Z}}\dim_\mathbb{K} H^a(T_\mathbf{E}^n(\Lambda))\exp(-at)
	\]
	for $\mathbf{E}\in\{\mathbf{X},\mathbf{Y}\}$. Since $\Lambda$ (and hence also $\Sigma^aT_\mathbf{E}^n(\Lambda)$) is a direct sum of $p+q$ indecomposable direct summands, \cite[Theorem 6.1]{broomhead} implies that \[\dim_{\mathbb{K}}\Hom_{\mathbf{D}^\mathrm{b}(\mod\Lambda)}(\Lambda, \Sigma^aT_\mathbf{E}^n(\Lambda))\leq 2(p+q)^2\] for a fixed number $a\in\mathbb{Z}$. By Corollary~\ref{exactly-linear-number-homologies} the number of possible values of $a$, such that $H^a(T_\mathbf{E}^n(\Lambda))\neq 0$, grows linearly with respect to $n$. Since Lemma~\ref{min-max-homologies-lemma}(2),(4) we also know the minimal and maximal indices of non-zero cohomologies in the objects $T_\mathbf{E}^n(\Lambda)$, we can now apply Lemma~\ref{limit-lemma}. This way, for $t\geq 0$ we have
	\begin{equation*}
		\begin{split}
			h_{T_\mathbf{X}}(t) &= \lim_{n\to\infty}\frac{1}{n}\cdot\left(-\left(-\frac{rn}{r+q}+\lambda_n\right)\right) t=\frac{r}{r+q}t,\\
			h_{T_\mathbf{Y}}(t) &= \lim_{n\to\infty}\frac{1}{n}\cdot 0\cdot t=0,
		\end{split}
	\end{equation*}
	where $\lambda_n\in(-r,r]$ for all $n>0$. Similarly, for $t\leq 0$, we have
	\begin{equation*}
		h_{T_\mathbf{X}}(t) = 0, \hspace{1cm} h_{T_\mathbf{Y}}(t)=\frac{r}{r-p}.\qedhere
	\end{equation*}
\end{proof}
\subsection{General functor}
\begin{proposition}\label{theorem-of-entropy-fin-gldim}
	Let $F$ be an autoequivalence of $\mathbf{D}^\mathrm{b}(\mod\Lambda)$. Then $F\equiv T_\mathbf{X}^kT_\mathbf{Y}^l\Sigma^s$ for some $k,l,s\in\mathbb{Z}$ and
	\[
		h_{F}(t)=st+\max\left(l\frac{r}{r-p}t, k\frac{r}{r+q}t\right).
	\]
\end{proposition}
\begin{proof}
	By Proposition~\ref{automorphisms-gldim-finite}(1) there exist $k,l,s\in\mathbb{Z}$ such that $F\equiv T_\mathbf{X}^kT_\mathbf{Y}^l\Sigma^s$. By Lemma~\ref{entropy-properties}(1) and Corollary~\ref{sigma-properties}(1) we have $h_F(t)=h_{T_\mathbf{X}^kT_\mathbf{Y}^l}(t)+st$. Therefore, we just need to calculate the entropy of autoequivalences of the form $T_\mathbf{X}^kT_\mathbf{Y}^l$ for some $k,l\in\mathbb{Z}$.

	Assume first that $F=T_\mathbf{X}^{k(r+q)}T_\mathbf{Y}^{l(r-p)}$ for some $k,l\in\mathbb{Z}$. We have $k\leq l$ or $l\leq k$. In the former case, we can write $l=k+m$ for some $m\geq 0$. Then we have the following relations
	\[
		F\equiv T_\mathbf{X}^{k(r+q)}T_\mathbf{Y}^{k(r-p)}T_\mathbf{Y}^{m(r-p)}\equiv\Sigma^{kr}T_\mathbf{Y}^{m(r-p)}.
	\]
	Using Lemmas~\ref{entropy-properties}(1),(2),~\ref{all-funcs-are-inverse-regular}, Corollary~\ref{sigma-properties}(1) and Proposition~\ref{single-power-entropy} we get
	\[
		h_F(t)=krt+h_{T_\mathbf{Y}^{m(r-p)}}(t)=krt-m(r-p)h_{T_\mathbf{Y}}(-t) = \begin{cases*}
			krt & if $t\leq 0$, \\
			lrt & if $t\geq 0$.
		\end{cases*}
	\]
	Similarly, if $l\leq k$, we show that
	\[
		h_F(t) = \begin{cases*}
			lrt & if $t\leq 0$, \\
			krt & if $t\geq 0$.
		\end{cases*}
	\]

	We can now study an autoequivalence of the form $F=T_\mathbf{X}^kT_\mathbf{Y}^l$ for arbitrary $k,l\in\mathbb{Z}$. Observe that by Lemma~\ref{entropy-properties}(2) we have
	\begin{equation*}
		\begin{split}
			h_F(t) = \frac{(r+q)(-r+p)}{(r+q)(-r+p)}h_F(t) = \frac{-1}{(r+q)(r-p)}h_{T_\mathbf{X}^{-k(r+q)(r-p)}T_\mathbf{Y}^{-l(r+q)(r-p)}}(t).
		\end{split}
	\end{equation*}
	Using the previous calculations, if $-k(r-p)\leq -l(r+q)$ (equivalently, $k(r-p)\geq l(r+q)$), then
	\[
		h_F(t) = \begin{cases*}
			k\frac{r}{r+q}t & $t\leq 0$, \\
			l\frac{r}{r-p}t & $t\geq 0$.
		\end{cases*}
	\]
	If $k(r-p)\leq l(r+q)$, we obtain
	\[
		h_F(t) = \begin{cases*}
			l\frac{r}{r-p}t & $t\leq 0$, \\
			k\frac{r}{r+q}t & $t\geq 0$.
		\end{cases*}
	\]
	and the claim follows.
\end{proof}
\subsection{Polynomial entropy}
\begin{proposition}\label{tx-ty-poly-entropy}
	For any $m\neq 0$, $t\in\mathbb{R}$ and $\mathbf{E}\in\{\mathbf{X},\mathbf{Y}\}$ we have the equality
	\[
		\hat{h}_{T_\mathbf{E}^m}(t)=\lim_{n\to\infty}\frac{\log(\delta_{\Lambda,T_\mathbf{E}^{nm}(\Lambda)}(t))-nh_{T_\mathbf{E}^{m}}( t)}{\log(n)}=\begin{cases*}
			1 & if $t=0$,\\
			0 & if $t\neq 0$.
		\end{cases*}
	\]
\end{proposition}
\begin{proof}
	We will show the calculations for $T_\mathbf{X}$ only as the proof for $T_\mathbf{Y}$ is done similarly. By Corollary~\ref{exactly-linear-number-homologies} we infer that
	\begin{equation}\label{delta-equation}
		\delta_{\Lambda,T_\mathbf{X}^n(\Lambda)}'(t) = \sum_{a\in A_n} \exp(-at) = \sum_{a\in A'_n}\exp(at),
	\end{equation}
	for some multisets $A_n$ whose cardinality grows linearly with respect to positive integer $n$ and $A'_n=\{-a\;|\; a\in A_n\}$. By Lemma~\ref{min-max-homologies-lemma}(3),(4) we know that $\min(A'_n)=0$ and $a_n\coloneqq \max(A'_n) = \frac{rn}{q+r}+\lambda_n$ for some $\lambda_n\in[-r,r)$.

	Assume first that $t\neq 0$. We have, by the fact that $\dim_{\mathbb{K}}\Hom_{\mathbf{D}^\mathrm{b}(\mod\Lambda)}(\Lambda, T_\mathbf{X}^n(\Lambda))\leq 2(p+q)^2$ (see \cite[Theorem 6.1]{broomhead}), the following bound
	\begin{equation*}
		\begin{split}
			\delta_{\Lambda,T_\mathbf{X}^n(\Lambda)}'(t)&\leq 2(p+q)^2\cdot\left(\sum_{i=0}^{a_n}\exp(it)\right) = 2(p+q)^2\cdot\left(\frac{1-\exp(t(a_n+1))}{1-\exp(t)}\right)\\
			&= 2(p+q)^2\cdot\left(\frac{1}{1-\exp(t)}-\frac{\exp(t)}{1-\exp(t)}\exp(a_nt)\right) \eqqcolon L.
		\end{split}
	\end{equation*}
	If $t<0$, then $1-\exp(t)>0$, hence
	\[
		L\leq \frac{2(p+q)^2}{1-\exp(t)} = k_1 \cdot \exp(0),
	\]
	where $k_1\coloneqq\frac{2(p+q)^2}{1-\exp(t)}$. Using that $h_{T_\mathbf{X}}(t)=0$ by Proposition~\ref{theorem-of-entropy-fin-gldim}, we now have the following inequalities
	\[
		\frac{\log(\exp(0))}{\log(n)}\leq \frac{\log(\delta_{\Lambda, T_\mathbf{X}^n(\Lambda)}'(t))-nh_{T_\mathbf{X}}(t)}{\log(n)}\leq \frac{\log(k_1\exp(0))}{\log(n)}.
	\]
	By applying the squeeze theorem we infer that the sequence in the middle converges to $0$. Thus, for $t<0$, by Lemma~\ref{poly-entropy-formula} we have $\hat{h}_{T_\mathbf{X}}(t)=0$. On the other hand, if $t>0$, then
	\[
		L\leq k_2\cdot\exp(a_nt),
	\]
	where $k_2\coloneqq \frac{2(p+q)^2\exp(t)}{\exp(t)-1}$. Therefore, we obtain the following inequalities
	\[
		\frac{\log(\exp(a_nt))-nh_{T_\mathbf{X}}(t)}{\log(n)}\leq \frac{\log(\delta_{\Lambda,T_\mathbf{X}^n(\Lambda)}'(t)) - nh_{T_\mathbf{X}}(t)}{\log(n)}\leq \frac{\log(k_2\exp(a_nt)) - nh_{T_\mathbf{X}}(t)}{\log(n)}.
	\]
	Again, by applying the squeeze theorem, we infer that the sequence in the middle converges to $0$ (note that $nh_{T_\mathbf{X}}(t)=\frac{rn}{q+r}t=a_nt-\lambda_nt$ by Proposition~\ref{theorem-of-entropy-fin-gldim}) and hence, for $t>0$ we have $\hat{h}_{T_\mathbf{X}}(t)=0$ by Lemma~\ref{poly-entropy-formula}.

	Observe now, that for $t=0$ the following holds
	\[
		\lim_{n\to\infty}\frac{\log(\delta_{\Lambda,T_\mathbf{X}^n(\Lambda)}'(0))-nh_{T_\mathbf{X}}(0)}{\log(n)} = \lim_{n\to\infty}\frac{\log(|A'_n|\cdot\exp(0))}{\log(n)} = 1,
	\]
	since $h_{T_\mathbf{X}}(0)=0$ by Proposition~\ref{theorem-of-entropy-fin-gldim}.

	Thus
	\[
		\lim_{n\to\infty}\frac{\log(\delta'_{\Lambda, T_\mathbf{X}^n(\Lambda)}(t))-nh_{T_\mathbf{X}}(t)}{\log(n)}=\begin{cases*}
			1 & if $t=0$,\\
			0 & if $t\neq 0$.
		\end{cases*}
	\]
	Moreover, we get by Lemma~\ref{poly-entropy-formula}(2) that for all $t\in\mathbb{R}$ we have
	\[
		\hat{h}_{T_\mathbf{X}}(t)=\lim_{n\to\infty}\frac{\log(\delta_{\Lambda, T_\mathbf{X}^n(\Lambda)}(t))-nh_{T_\mathbf{X}}(t)}{\log(n)}.
	\]
	Consequently, by Lemma~\ref{cat-poly-powers} we deduce that for any $m>0$ and all $t\in\mathbb{R}$ we have
	\[
		\hat{h}_{T_\mathbf{X}}(t)=\hat{h}_{T_\mathbf{X}^{m}}(t)=\lim_{n\to\infty}\frac{\log(\delta_{\Lambda, T_\mathbf{X}^{nm}(\Lambda)}(t))-nh_{T_{\mathbf{X}}^{m}}(t)}{\log(n)}.
	\]
	Finally, by Lemma~\ref{poly-ent-properties}(1) we infer that
	\[
		\hat{h}_{T_\mathbf{X}^{-1}}(t)=\lim_{n\to\infty}\frac{\log(\delta'_{\Lambda, T_\mathbf{X}^{-n}(\Lambda)}(t))-nh_{T_{\mathbf{X}^{-1}}}(t)}{\log(n)}=\hat{h}_{T_\mathbf{X}}(-t).
	\]
	Using Lemmas~\ref{poly-entropy-formula}(2) and~\ref{cat-poly-powers} again we can now conclude that for any $m>0$ and all $t\in\mathbb{R}$ we have
	\[
		\hat{h}_{T_\mathbf{X}^{-1}}(t)=\hat{h}_{T_\mathbf{X}^{-m}}(t)=\lim_{n\to\infty}\frac{\log(\delta_{\Lambda, T_\mathbf{X}^{-nm}(\Lambda)}(t))-nh_{T_{\mathbf{X}}^{-m}}(t)}{\log(n)}.\qedhere
	\]
\end{proof}

We will now calculate the polynomial entropy of any autoequivalence $F$ of $\mathbf{D}^\mathrm{b}(\mod\Lambda)$.
\begin{proposition}\label{theorem-for-poly-entropy-fin-gldim}
	Let $F$ be an autoequivalence of $\mathbf{D}^\mathrm{b}(\mod\Lambda)$. Then $F\equiv T_\mathbf{X}^kT_\mathbf{Y}^l\Sigma^s$ for some $k,l,s\in\mathbb{Z}$ and \[
		\hat{h}_F(t)=\begin{cases*}
			1 & if $l(r+q)\neq k(r-p)$ and $t=0$,\\
			0 & otherwise.
		\end{cases*}
		\]
\end{proposition}
\begin{proof}
	By Proposition~\ref{automorphisms-gldim-finite}(1) we know that there exist $k,l,s\in\mathbb{Z}$ such that $F\equiv T_\mathbf{X}^kT_\mathbf{Y}^l\Sigma^s$. By Lemma~\ref{poly-entropy-on-equivalent-functors} and Corollary~\ref{sigma-properties}(2) we infer that $\hat{h}_F(t)=\hat{h}_{T_\mathbf{X}^kT_\mathbf{Y}^l}(t)$ for all $t\in\mathbb{R}$. Therefore, we may assume that $F=T_\mathbf{X}^kT_\mathbf{Y}^l$.

	Observe first, that we have $\Sigma^{kr}\equiv T_\mathbf{X}^{k(r+q)}T_\mathbf{Y}^{k(r-p)}$, or equivalently, $T_\mathbf{X}^{k(r+q)}\equiv \Sigma^{kr}T_\mathbf{Y}^{-k(r-p)}$. Consequently, we have the following
	\[
		F^{r+q} \equiv T_\mathbf{X}^{k(r+q)}T_\mathbf{Y}^{l(r+q)} \equiv \Sigma^{kr}T_\mathbf{Y}^{l(r+q)-k(r-p)},
	\]
	and thus by Proposition~\ref{sigma-properties}(2) we obtain $\hat{h}_{F^{r+q}}(t)=\hat{h}_{T_\mathbf{Y}^{l(r+q)-k(r-p)}}(t)$.

	If $l(r+q)=k(r-p)$, then $F^{r+q}\equiv \id_{\mathbf{D}^\mathrm{b}(\mod\Lambda)}=\Sigma^0$ and by Lemmas~\ref{entropy-of-sigma}(2),~\ref{cat-poly-powers-inverse} and Corollary~\ref{all-deltas-are-positive} we deduce that $\hat{h}_F(t)=\hat{h}_{F^{r+q}}(t)=0$.
	
	Suppose now that $l(r+q)\neq k(r-p)$. By Proposition~\ref{tx-ty-poly-entropy} we have
	\[
	\hat{h}_{F^{r+q}}(t)=\lim_{n\to\infty}\frac{\log(\delta_{\Lambda,T_\mathbf{Y}^{n(l(r+q)-k(r-p))}(\Lambda)}(t))-nh_{T_\mathbf{Y}^{l(r+q)-k(r-p)}}( t)}{\log(n)}=\begin{cases*}
		1 & if $t=0$,\\
		0 & if $t\neq 0$.
	\end{cases*}
	\]
	By Lemma~\ref{limit-of-functor-induces-another-limit} we deduce that
	\begin{equation*}
		\begin{split}
			\hat{h}_{F^{r+q}}(t)= \lim_{n\to\infty}\frac{\log(\delta_{\Lambda,F^{n(r+q)}(\Lambda)}(t))-nh_{F^{r+q}}(t)}{\log(n)}.
		\end{split}
	\end{equation*}
	Moreover, Corollary~\ref{all-deltas-are-positive} allows us to infer that $\delta_{\Lambda, F^k(\Lambda)}(t)>0$ for all $k\in\mathbb{Z}$ and $t\in\mathbb{R}$. Therefore, by applying Lemma~\ref{cat-poly-powers-inverse} we infer that $\hat{h}_{F}(t) = \hat{h}_{F^{r+q}}(t)$.
\end{proof}
\begin{proof}[Proof of Theorem~\ref{theorem-finite-gldim}]
	Assume that $\Lambda$ is a derived discrete $\mathbb{K}$-algebra of finite global dimension which is not piecewise hereditary of Dynkin type and $F$ is a derived autoequivalence of $\Lambda$. Suppose that $\Lambda$ is derived equivalent to $\Lambda(p,q,r)$ with $q\geq 0$, $1\leq r<p$.
	
	By Proposition \ref{automorphisms-gldim-finite-arbitrary} we infer that $F\equiv T_{\mathbf{X},\Lambda}^kT_{\mathbf{Y},\Lambda}^l\Sigma^s$. Moreover, by Lemmas \ref{entropy-of-adjunction},  \ref{polynomial-entropy-of-adjunction} and Proposition \ref{automorphisms-gldim-finite} we deduce that $h_F(t)=h_{F_1}(t)$ and $\hat{h}_F(t)=\hat{h}_{F_1}(t)$ for a derived autoequivalence $F_1$ of $\Lambda(p,q,r)$ such that $F_1\equiv T_\mathbf{X}^kT_\mathbf{Y}^l\Sigma^s$. Now we use Propositions~\ref{theorem-of-entropy-fin-gldim} and~\ref{theorem-for-poly-entropy-fin-gldim} and finish the proof.
\end{proof}

\section{Infinite global dimension}\label{section-4}
In this section we prove Theorem~\ref{theorem-infinite-gldim}.

We will start with derived discrete algebras of the form $\Lambda(p,q,p)$. By Proposition~\ref{automorphisms-gldim-infinite}(1) for any autoequivalence $F$ of $\mathbf{D}^\mathrm{b}(\mod\Lambda(p,q,p))$ there exist $k,s\in\mathbb{Z}$ such that $F\equiv T_\mathbf{X}^k\Sigma^s$. By using Lemma~\ref{entropy-properties}(1) and Corollary~\ref{sigma-properties}(1), we infer that $h_F(t)=h_{T_\mathbf{X}^k\Sigma^s}(t)=st+h_{T_\mathbf{X}^k}(t)$ for all $t\in\mathbb{R}$.

Observe that by Proposition~\ref{automorphisms-gldim-infinite}(2) we have $\Sigma^p\equiv T_\mathbf{X}^{p+q}$, hence by Lemma~\ref{entropy-properties}(1),(2) and Proposition~\ref{entropy-of-sigma}
\begin{equation*}
	\begin{split}
		h_{\Sigma^p}(t) &= h_{T_\mathbf{X}^{p+q}}(t),\\
		pt &= (p+q)h_{T_\mathbf{X}}(t),\\
		\frac{p}{p+q}t &= h_{T_\mathbf{X}}(t),
	\end{split}
\end{equation*}
for all $t\in\mathbb{R}$. Similarly, $\Sigma^{-p}\equiv T_\mathbf{X}^{-p-q}$ and therefore $h_{T_\mathbf{X}^{-1}}(t)=\frac{-p}{p+q}t$. We conclude that for any $k\in\mathbb{Z}$ we have $h_{T_\mathbf{X}^k}(t)=k\frac{p}{p+q}t$.

As for the polynomial entropy of an arbitrary autoequivalence $F\equiv T_\mathbf{X}^k\Sigma^s$, similarly to the finite global dimension case, we conclude by Lemma~\ref{poly-entropy-on-equivalent-functors} that $\hat{h}_F(t)=\hat{h}_{T_\mathbf{X}^k\Sigma^s}(t)$. Moreover, $\hat{h}_{T_\mathbf{X}^k\Sigma^s}(t)=\hat{h}_{T_\mathbf{X}^k}(t)$ by Corollary~\ref{sigma-properties}(2).

Let $G\in\mod\Lambda(p,q,p)$ be a split-generator of $\mathbf{D}^\mathrm{b}(\mod\Lambda(p,q,p))$ and let $F_1\coloneqq T_\mathbf{X}^{k(p+q)}$. Then by the relation $T_\mathbf{X}^{p+q}\equiv\Sigma^p$, Lemmas~\ref{delta-properties}(1) and~\ref{delta-value-of-sigma} we have \[\delta_{G,T_\mathbf{X}^{nk(p+q)}(G)}(t)=\delta_{G, \Sigma^{nkp}G}(t)=\exp(nkpt)\]
for all $n\geq 0$ and $t\in\mathbb{R}$. Observe that in this case
\[
	\hat{h}_{F_1}(t)=\hat{h}_{\Sigma^{kp}}(t)=0=\lim_{n\to\infty}\frac{\log(\exp(nkpt))-nkpt}{\log(n)}=\lim_{n\to\infty}\frac{\log(\delta_{G,T_\mathbf{X}^{nk(p+q)}G}(t))-nh_{F_1}(t)}{\log(n)}.
\]
Thus, by Corollary~\ref{all-deltas-are-positive} and Lemma~\ref{cat-poly-powers-inverse} we conclude that $\hat{h}_{T_\mathbf{X}^k}(t)=0$.

Using the same argument as in the proof of Theorem~\ref{theorem-finite-gldim}, we obtain the desired result for an arbitrary derived discrete algebra $\Lambda$ of infinite global dimension.

We can prove similar results for the category $\per(\Lambda)$. Since proof is done in the same way as for $\mathbf{D}^\mathrm{b}(\mod\Lambda)$, we leave it to the reader.\qed


\begin{thebibliography}{00}
	
	\bibitem{avella-alaminos-geiss}
	D. Avella-Alaminos and Ch. Geiss, \textit{Combinatorial derived invariants for gentle algebras}, J. Pure Appl.
	Algebra 212 (2008), no. 1, 228–243, DOI 10.1016/j.jpaa.2007.05.014
	
	\bibitem{bobinski}
	G. Bobi\'nski, Ch. Geiss, and A. Skowro\'nski, \textit{Classification of discrete derived categories}, Cent. Eur. J.
	Math. 2 (2004), no. 1, 19–49, DOI 10.2478/BF02475948.
	
	\bibitem{broomhead}
	N. Broomhead, D. Pauksztello, and D. Ploog, \textit{Discrete derived categories I: homomorphisms, autoequivalences and t-structures}, Math. Z. 285 (2017), no. 1-2, 39–89, DOI 10.1007/s00209-016-1690-1.
	
	\bibitem{haiden}
	G. Dimitrov, F. Haiden, L. Katzarkov, and M. Kontsevich, \textit{Dynamical systems and categories}, The
	influence of Solomon Lefschetz in geometry and topology, Contemp. Math., vol. 621, Amer. Math. Soc.,
	Providence, RI, 2014, pp. 133–170, DOI 10.1090/conm/621/12421.
	
	\bibitem{fan-fu-ouchi}
	Y.-W. Fan, L. Fu, and G. Ouchi, \textit{Categorical polynomial entropy}, Adv. Math. 383 (2021), Paper No.
	107655, 50, DOI 10.1016/j.aim.2021.107655.
	
	\bibitem{geiss-reiten}
	Ch. Geiss and I. Reiten, \textit{Gentle algebras are Gorenstein}, Representations of algebras and related topics,
	Fields Inst. Commun., vol. 45, Amer. Math. Soc., Providence, RI, 2005, pp. 129–133.
	
	\bibitem{hochenegger-kalck-ploog}
	A. Hochenegger, M. Kalck and D. Ploog, \textit{Spherical subcategories in algebraic geometry}, Math. Nachr. {\bf 289} (2016), no.~11-12, 1450--1465
	
	\bibitem{huybrechts}
	D. Huybrechts, \textit{Fourier-Mukai transforms in algebraic geometry}, Oxford Mathematical Monographs,
	The Clarendon Press, Oxford University Press, Oxford, 2006.
	
	\bibitem{ikeda}
	A. Ikeda, \textit{Mass growth of objects and categorical entropy}, Nagoya Math. J. {\bf 244} (2021), 136--157, DOI 10.1017/nmj.2020.9
	
	\bibitem{katok}
	A. Katok and B. Hasselblatt, \textit{Introduction to the modern theory of dynamical systems}, Encyclopedia of
	Mathematics and its Applications, vol. 54, Cambridge University Press, Cambridge, 1995.
	
	\bibitem{kikuta}
	K. Kikuta, Y. Shiraishi, and A. Takahashi, \textit{A note on entropy of auto-equivalences: lower bound and the
	case of orbifold projective lines}, Nagoya Math. J. 238 (2020), 86–103, DOI 10.1017/nmj.2018.21.
	
	\bibitem{krause}
	H. Krause, \textit{Homological theory of representations}, Cambridge Studies in Advanced Mathematics,
	vol. 195, Cambridge University Press, Cambridge, 2022.
	
	\bibitem{reiten}
	I. Reiten and M. Van den Bergh, \textit{Noetherian hereditary abelian categories satisfying Serre duality}, J.
	Amer. Math. Soc. 15 (2002), no. 2, 295–366, DOI 10.1090/S0894-0347-02-00387-9.
	
	\bibitem{rickard}
	J. Rickard, \textit{Morita theory for derived categories}, J. London Math. Soc. (2) 39 (1989), no. 3, 436–456,
	DOI 10.1112/jlms/s2-39.3.436.
	
	\bibitem{vossieck}
	D. Vossieck, \textit{The algebras with discrete derived category}, J. Algebra 243 (2001), no. 1, 168–176, DOI
	10.1006/jabr.2001.8783.
	
\end{thebibliography}
\end{document}